\documentclass[10pt]{amsart}

\usepackage{psfrag}
\usepackage{amsmath}
\usepackage{amssymb}
\usepackage[english]{babel}
\usepackage{amsthm}
\usepackage{amsfonts}
\usepackage[applemac]{inputenc}
\usepackage{tikz}
\usetikzlibrary{arrows}
\usetikzlibrary{patterns}
\usetikzlibrary{shapes}
\usepackage{mathrsfs}

\newtheorem{teo}{Theorem}
\newtheorem{lemma}[teo]{Lemma}
\newtheorem{defi}[teo]{Definition}
\newtheorem{coro}[teo]{Corollary}
\newtheorem{propo}[teo]{Proposition}
\newtheorem*{propos}{Proposition}
\newtheorem*{thma}{Theorem A}
\newtheorem*{thmb}{Theorem B}
\newtheorem*{thmc}{Theorem C}
\newtheorem*{thmd}{Theorem D}

\newcommand{\R}{\mathbb{R}}
\newcommand{\ii}{\mathrm{i}}
\newcommand{\ind}{\textrm{ind}}

\newcommand{\eps}{\epsilon}
\newcommand{\N}{\mathbb{N}}
\newcommand{\C}{\mathbb{C}}
\newcommand{\CP}{\mathbb{C}\textrm{P}}
\newcommand{\RP}{\mathbb{R}\mathrm{P}}
\newcommand{\HP}{\mathbb{H}\mathrm{P}}
\newcommand{\Z}{\mathbb{Z}}
\newcommand{\D}{\mathcal{D}}

\renewcommand{\P}{\mathbb{P}}
\renewcommand{\L}{\mathcal{L}}
\newcommand{\Q}{\mathcal{Q}}

\theoremstyle{remark}
\newtheorem{remark}[]{Remark}
\newtheorem{fact}[]{Fact}
\newtheorem{example}[]{Example}
\title{Systems of Quadratic Inequalities}

\author{A. Agrachev}
\thanks{SISSA, Trieste \& Steklov Math. Inst.,
Moscow}

\author{ A. Lerario}
\thanks{SISSA, Trieste}

\begin{document}

\maketitle

\begin{abstract}
We present a spectral sequence which efficiently
computes Betti numbers of a closed semi-algebraic subset of
${\RP}^n$ defined by a system of quadratic inequalities and the
image of the homology homomorphism induced by the inclusion of
this subset in $\RP^n$. We do not restrict ourselves to the term
$E_2$ of the spectral sequence and give a simple explicit formula
for the differential $d_2$.
\end{abstract}

\section{Introduction}

In this paper we study closed semialgebraic subsets of ${\RP}^n$
presented as the sets of solutions of systems of homogeneous
quadratic inequalities. Systems are arbitrary: no regularity
condition is required and systems of equations are included as
special cases. Needless to say, standard Veronese map reduces any
system of homogeneous polynomial inequalities to a system of
quadratic ones (but the number of
inequalities in the system increases). The nonhomogeneous affine case will be
the subject of another publication.

To study a system of quadratic inequalities we focus on the dual
object. Namely, we take the convex hull, in the space of all real
quadratic forms on $\R^{n+1}$, of those quadratic forms involved
in the system, and we try to recover the homology of the set of
solutions from the arrangement of this convex hull with respect to
the cone of degenerate forms. This approach allows to efficiently
compute Betti numbers of the set of solutions for a very big number
of variables $n$ as long as the number of linearly independent
inequalities is limited. Moreover, this approach works well for
systems of integral quadratic inequalities (i.\,e. in the infinite
dimension, far beyond the semi-algebraic context) as we plan to
prove in another paper.

Let $p:\R^{n+1}\to\R^{k+1}$ be a homogeneous quadratic map and
$K\subset\R^{k+1}$ a convex polyhedral cone in $\R^k$ (zero cone
$K=\{0\}$ is permitted). We are going to study the semialgebraic
set
$$
X_p=\{\bar x=(x_0:\ldots:x_n)\in{\RP}^n \mid p(x_0,\ldots,x_n)\in
K\}.
$$
More precisely, we are going to compute the homology
$H_*(X_p;\Z_2)$ and the image of the map $\iota_*:H_*(X_p;\Z_2)\to
H_*({\RP}^n;\Z_2)$, where $\iota:X_p\to{\RP}^n$ is the inclusion.

In what follows, we use shortened notations
$H_*(X_p;\Z_2)=H_*(X_p),\ {\RP}^n=\P^n$.

Let $\mathcal Q$ be the space of real quadratic forms on $\mathbb
R^{n+1}$. Given $q\in\mathcal Q$, we denote by $\ii^+(q)\in\N$ the
positive inertia index of $q$ that is the maximal dimension of a
subspace of $\R^{n+1}$ where the form $q$ is positive definite.
Similarly, $\ii^-(q)\doteq\ii^+(-q)$ is the negative inertia
index. We set:
$$
\mathcal Q^j=\{q\in\mathcal Q : \ii^+(q)\ge j\}.
$$
We denote by $\bar p:{\R^{k+1}}^*\to\mathcal Q$ the linear systems
of quadratic forms associated to the map $p$. In coordinates:
$$
p=\left(\begin{smallmatrix}p^0\\
\vdots\\p^k\end{smallmatrix}\right),\quad p^i\in\mathcal Q,\qquad
\bar p(\omega)=\omega p=\sum\limits_{i=0}^k\omega_ip^i,\quad
\forall\,\omega=(\omega_0,\ldots,\omega_k)\in{\R^{k+1}}^*.
$$

More notations:
$$
K^\circ=\{\omega\in{\R^{k+1}}^* : \langle\omega,y\rangle\le 0,\
\forall y\in K\},\ \mathrm{the\ dual\ cone\ to}\ K;
$$
$$
\Omega=K^\circ\cap S^k=\{\omega\in K^\circ : |\omega|=1\};
$$
$$
C\Omega=K^\circ\cap B^{k+1}=\{\omega\in K^\circ : |\omega|\le 1\};
$$
$$
\Omega^j=\{\omega\in\Omega : \ii^+(\omega p)\ge j\}.
$$

\begin{thma} There exists a cohomological spectral sequence of the
first quadrant $(E_r,d_r)$ converging to $H_{n-*}(X)$ such that
$E_2^{ij}=H^i(C\Omega,\Omega^{j+1})$.
\end{thma}

We define $\mu\doteq\max\limits_{\eta \in\Omega}\ii^{+}(\eta).$ If
$\mu=0$ then $X_p=\P^n$; otherwise we can describe the term
$E_{2}$ by the following table where cohomology groups are
replaced with isomorphic ones according to the long exact sequence
of the pair $(C\Omega, \Omega^{j+1})$.\\

$$
\begin{array}{c|c|c|c|c|c|c|c|c|}
&0&0&0&&&&&\\
n&\Z_{2}&0&0&&&&&\\
%&&&&&&&\\
&\vdots&\vdots&\vdots&&&&&\\
%&&&&&&&\\
\mu&\Z_{2}&0&0&\cdots&0&\cdots&0&0\\
&0&H^{0}(\Omega^{\mu})/\Z_{2}&H^{1}(\Omega^{\mu})&\cdots&H^{i}(\Omega^{\mu})&\cdots&H^{k}(\Omega^{\mu})&0\\
%&&&&&&&\\
&\vdots&\vdots&\vdots&&\vdots&&\vdots&\vdots\\
%&&&&&&&\\
&0&H^{0}(\Omega^{j+1})/\Z_{2}&H^{1}(\Omega^{j+1})&\cdots&H^{i}(\Omega^{j+1})&\cdots&H^{k}(\Omega^{j+1})&0\\
%&&&&&&&\\
&\vdots&\vdots&\vdots&&\vdots&&\vdots&\vdots\\
%&&&&&&&\\
&0&H^{0}(\Omega^{1})/\Z_{2}&H^{1}(\Omega^{1})&\cdots&H^{i}(\Omega^{1})&\cdots&H^{k}(\Omega^{1})&0\\
\hline
\end{array}
$$
\\\\
\begin{example}Let $n=k=2,\ p(x_0,x_1,x_2)=\left(\begin{smallmatrix}x_0x_1\\
x_0x_2\\x_1x_2\end{smallmatrix}\right),\ K=\{0\}$. Then
$$\Omega=\Omega^1=S^2,\quad \Omega^2=\{\omega\in S^2 :
\omega_0\omega_1\omega_2<0\},\quad \Omega^{3}=\emptyset.$$ The term
$E_2$ has the form:
$$
\begin{array}{|cccc}
\Z_2&0&0&0\\

0&({\Z_2})^3&0&0\\

0&0&0&\Z_2\\
\hline
\end{array}
$$
In this case $d_{2}:(\Z_{2})^{3}\to \Z_{2}$ is a non-vanishing differential and the set $X_p$ consists of 3 points.
\end{example}
 Let $\mathscr{G}_j=\left\{(V,q)\in Gr(j)\times\left(\mathcal
Q^j\setminus\mathcal Q^{j+1}\right) : q\bigr|_V>0\right\},$
where $Gr(j)$ is the Grassmannian of $j$-dimensional subspaces of
$\R^{n+1}$. It is easy to see that the projection $
\pi:(V,q)\mapsto q,\ (V,q)\in\mathscr{G}_j$ is a homotopy
equivalence.

Let us consider the tautological vector bundle $\mathcal{V}_j$ over
$\mathscr{G}_j$ whose fiber at $(V,q)\subset\mathscr{G}_j$ is the
space $V\in\R^{n+1}$ and the first Stiefel--Whitney class of this
bundle $w_1(\mathcal V_j)\in H^1(\mathscr{G}_j)$. Recall that
$w_1(\mathcal V_j)$ vanishes at a curve $f:S^1\to\mathscr{G}_j$ if
and only if $f^*\mathcal V_j$ is a trivial bundle. Moreover, the
value of $w_1(\mathcal V_j)$ at $f$ depends only on the curve
$\pi\circ f$ in $\mathcal Q^j\setminus\mathcal Q^{j+1}$ and
$w_1(\mathcal V_j)=\pi^*\nu_j$ for a well-defined class $\nu_j\in
H^1\left(\mathcal Q^j\setminus\mathcal Q^{j+1}\right)$.

\begin{propos} The differentials $d_2$, of the spectral sequence
$(E_r,d_r)$ is determined by the class $\bar
p\bigr|_{\Omega^j\setminus\Omega^{j+1}}^*(\nu_j)\in
H^1(\Omega^j\setminus\Omega^{j+1})$. If $\bar
p\bigr|_{\Omega^j\setminus\Omega^{j+1}}^*(\nu_j)=0,\
\forall\,j>0$, then $E_3=E_2$.
\end{propos}

The classes $\nu_j$ are defined without any use of the Euclidean
structure on $\R^{n+1}$. This structure is however useful for the
explicit calculation of $d_2$. Given $q\in\mathcal Q$, let
$\lambda_1(q)\ge\cdots\ge\lambda_{n+1}(q)$ be the eigenvalues of the
symmetric operator $Q$ on $\R^{n+1}$ defined by the formula
$q(x)=\langle Qx,x\rangle,\ x\in\R^{n+1}$. Then $\mathcal
Q^j=\{q\in\mathcal Q : \lambda_j>0\}$. We set $\mathcal D
_j=\{q\in\mathcal Q : \lambda_j(q)\ne\lambda_{j+1}(q)\}$ and
denote by $\mathcal L^+_j$ the $j$-dimensional vector bundle over
$\D_j$ whose fiber at a point $q\in\Lambda_j$ equals
$span\{x\in\R^{n+1} : Qx=\lambda_ix,\ 1\le i\le j\}$. Obviously,
$\mathcal Q^j\setminus\mathcal Q^{j+1}\subset\mathcal D_j$ and
$\nu_j=w_1\left(\mathcal L^+_j\right)\bigr|_{\mathcal
Q^j\setminus\mathcal Q^{j+1}}$.

Now we set $\phi_j=\partial^* w_1(\mathcal L^+_j)$, where
$\partial^*:H^1(\mathcal D_j)\to H^2(\mathcal Q,\mathcal D_j)$ is
the connecting isomorphism in the exact sequence of the pair
$(\mathcal Q,\mathcal D_j)$. Recall that
$$
\mathcal Q\setminus\mathcal D_j=\{q\in\mathcal Q :
\lambda_j(q)=\lambda_{j+1}(q)\}
$$
is a codimension 2 algebraic
subset of $\mathcal Q$ whose singular locus
$$
sing\left(\mathcal Q\setminus\mathcal D_j\right)=\{q\in\mathcal Q
: (\lambda_{j-1}(q)=\lambda_{j+1}(q)) \vee
(\lambda_j(q)=\lambda_{j+2}(q))\}
$$
has codimension 5 in $\mathcal Q$. Let $f:B^2\to\mathcal Q$ be a
continuous map defined on the disc $B^2$ and such that $f(\partial
B^2)\subset\mathcal D_j$; the value of $\phi_j\in H^2(\mathcal
Q,\mathcal D_j)$ at $f$ equals the intersection number (modulo 2)
of $f$ and $\mathcal Q\setminus\mathcal D_j$.

\begin{thmb}[the differentials $d_2$] We have:
$$
d_2(x)=\left(x \smile\bar
p^*\phi_j\right)\bigr|_{(C\Omega,\Omega^j)}, \quad \forall\,x\in
H^*(C\Omega,\Omega^{j+1}),
$$
where $\smile$ is the cohomological product.
\end{thmb}

\begin{thmc} Let $(\iota_*)_{a}:H_a(X_p)\to H_a(\P^n),\
0\le a\le n$, be the homomorphism induced by the inclusion
$\iota: X_p\to\P^n$. Then $\mathrm{rk}(\iota_*)_{a}=\dim
E_\infty^{0,n-a}$.
\end{thmc}

Next theorem about hyperplane sections is a step towards the
understanding of functorial properties of the duality between the
semi-algebraic sets $X_p$ and the index functions $\mathrm
i^+\circ\bar p$.

Let $V$ be a codimension one subspace of $\R^{n+1}$ and $\bar
V\subset \RP^n$ the projectivization of $V$.
%Let $\ii_{V}^{+}:\Omega
%\to \N$ be the function $\omega \mapsto \ii^{+}((\omega p)_{|V});$
We define for $j>0$ the following sets:
$$
\Omega_V^{j}=\{\omega \in \Omega : \mathrm i^{+}\left(\omega p |_V\right)\geq j\}$$%(by definition)
%$$\Omega_{V}^{0}\backslash A^{0}\doteq C\Omega.$$
%With the previous conventions we prove the following theorem.
\begin{thmd}
There exists a cohomology spectral sequence $(G_{r},d_{r})$ of the
first quadrant converging to $H_{n-*}(X_p, X_p\cap \bar V)$ such
that
$$
G_{2}^{i,j}=H^{i}(\Omega_V^{j},\Omega^{j+1}),\
j>0, \quad G_{2}^{i,0}=H^{i}(C\Omega,\Omega^1).$$
\end{thmd}

Theorem A is proved in Section 2, the differential $d_2$ is
computed in Sections 3, 4, Theorem~3 on the imbedding to $\RP^n$
is proved in Section~5, and Theorem~D  on the hyperplane sections
in Section~6. In Section~7 we study a special case of the constant
index function where higher differentials can be easily computed
and consider some other examples.

\medskip Let us indicate the main general ideas these proofs are based on.

\smallskip\noindent {\bf Regularization.} Polynomial inequalities can be easily
regularized without change of the homotopy type of the space of
solutions. Indeed, given a polynomial $a$, the space of solutions
of the inequality $a(x)\le 0$ is a deformation retract of the
space of solutions of the inequality $a(x)\le\varepsilon$ for any
sufficiently small $\varepsilon>0$, and the inequality
$a(x)\le\varepsilon$ is regular for any $\varepsilon$ from the
complement of a discrete subset of $\R$. The regularization of the
equation $a(x)=0$ is a system of inequalities $\pm
a(x)\le\varepsilon$.

\medskip\noindent {\bf Duality.} The map $\bar p:K^\circ\to\mathcal
Q$ is the dual object to $X_p$. Moreover, $\P^n\setminus X_p$ is
homotopy equivalent to $B=\{(\omega,x\in\Omega\times\P^n : (\omega p)(x)>0\}$. For a regular system of quadratic inequalities,
our spectral sequence is the relative Leray spectral system of the
map $(\omega,x)\mapsto\omega$ applied to the pair
$(\Omega\times\P^n,B)$.

\medskip\noindent {\bf Localization.} Given $\omega_0$, we have:
 $B_{O_{\omega_0}}\approx B_{\omega_0}\approx\P^{\ii^+(\omega_{0} p)-1}$, where
$O_{\omega_0}$ is any sufficiently small contractible neighborhood
of $\omega_0$ and $\approx$ is the homotopy equivalence. This fact
allows to compute the member $E_2$ of the spectral sequence.

\medskip\noindent {\bf Regular homotopy.} This is perhaps the most
interesting tool which allows to compute the differential $d_2$.
The notion of regular homotopy is based on the dual
characterization for the regularity of a system of quadratic
inequalities. We say that the system defined by the map $p$ and
cone $K$ is regular if $p\bigr|_{\R^{n+1}\setminus\{0\}}$ is
transversal to $K$; in other words, if
$\mathrm{im}(D_xp)+T_{p(x)}K=\R^{k+1},\
\forall\,x\in\R^{n+1}\setminus\{0\}$ such that $p(x)\in K$.

The dual characterization of regularity concerns the linear map
$\bar p:\Omega\to\mathcal Q$ but can be naturally extended to any
smooth map $f:\Omega\to\mathcal Q$. Note that $\mathcal Q$ is the
dual space to $\R^{n+1}\odot\R^{n+1}$, the symmetric square of
$\R^{n+1}$. Let
$$
\mathcal Q_0=\{q\in\mathcal Q : \ker q\ne 0\},
$$
the discriminant of the space of quadratic forms. Then $\mathcal
Q_0$ is an algebraic hypersurface and
$$
\mathrm{sing}\mathcal Q_0=\{q\in\mathcal Q_0 : \dim\ker q>1\}.
$$
Given $q\in\mathcal Q_0\setminus\mathrm{sing}\mathcal Q_0$ and
$x\in\ker q\setminus 0$, the vector $x\odot x\in\mathcal Q^*$ is
normal to the hypersurface $\mathcal Q_0$ at $q$. We define a
co-orientation of $\mathcal Q_0\setminus\mathrm{sing}\mathcal Q_0$
by the claim that $x\odot x$ is a positive normal. For any, maybe
singular, $q\in\mathcal Q_0$ we define the {\it positive normal
cone} as follows:
$$
N^+_q=\{x\odot x : x\in\ker q\setminus 0\}.
$$
The cone $N^+_q$ consists of the limiting points of the sequences
$N^+_{q_i},\ i\in\N,$ where $q_i\in\mathcal Q_0\setminus
\mathrm{sing}\mathcal Q_0$ and $q_i\to q$ as $i\to\infty$.

We say that $f:\Omega\to\mathcal Q$ is {\it not regular} (with
respect to $\mathcal Q_0$) at $\omega\in\Omega$ if
$f(\omega)\in\mathcal Q_0$ and $\exists\,y\in N^+_\omega$ such
that $\langle D_\omega fv,y\rangle\le 0,\ \forall\,v\in
T_\omega\Omega$. The map $f$ is regular if it is regular at any
point. It is easy to check that the transversality of the
quadratic map $p\bigr|_{\R^{n+1}\setminus\{0\}}$ to the cone $K$
is equivalent to the regularity of the linear map $\bar
p:\Omega\to\mathcal Q$ where, we remind, $\Omega=K^\circ\cap S^k$.

A homotopy $f_t:\Omega\to\mathcal Q,\ 0\le t\le 1,$ is a {\it
regular homotopy} if all $f_t$ are regular maps. The following
fundamental geometric fact somehow explains the results of this
paper and gives a perspective for further research. If linear maps
$\bar p_0,\bar p_1$ are regularly homotopic then the pairs $(\P^n,\P^n\setminus X_{p_0})$ and
             $(\P^n,\P^n\setminus X_{p_1})$ are homotopy equivalent. Note that the maps $f_t$ in the
homotopy connecting $\bar p_0$ and $\bar p_1$ are just smooth, not
necessary linear. It is important that the cones $N^+_q,\ q\in
\mathrm{sing}\mathcal Q_0,$ are not convex. If $N^+_q$ would be
convex then regular homotopy would preserve the term $E_2$ of our
spectral sequence, the differentials $d_r,\ r\ge 2$, would vanish
and $E_2$ would be equal to $E_\infty$.

Regular homotopy was introduced in paper \cite{Agrachev2} devoted
to regular quadratic maps. In the mentioned paper, the term $E_2$
and the differential $d_2$ of a converging to the homology of the
double covering of $X_p$ spectral sequence were computed. Again
for a regular quadratic map, all the differentials of a spectral
sequence converging to $H^*(\P^n\setminus X_p)$ were announced
(without proof) in \cite{agr}. We have to confess that,
unfortunately, only the differential $d_2$ was computed correctly.
Universal upper bounds for the Betti numbers of the sets defined
by systems of quadratic inequalities or equations were obtained in
\cite{Ba,BaKe,BaPaRo,DeItKh}.

\medskip\noindent {\sl Remark.} An Hermitian quadratic form is a
quadratic form $q:\C^{n+1}\to\R$ such that $q(iz)=q(z)$.
Similarly, a ``quaternionic'' quadratic form is a quadratic form
$q:\mathbb H^{n+1}\to\R$ such that $q(iw)=q(jw)=q(w).$ There are
obvious Hermitian and ``quaternionic" versions of the theory
developed in this paper (for systems of Hermitian or
``quaternionic" quadratic inequalities). You simply substitute
$\RP^n$ with $\CP^n$ or $\HP^n$, Stiefel--Whitney classes with
Chern or Pontryagin classes, differentials $d_r$ with
differentials $d_{2r-1}$ or $d_{4r-3}$, and compute homology with
coefficients in $\Z$ instead of $\Z_2$ (see also \cite{agr}).

\section{The spectral sequence}
Using the above notations for $\Omega=K^{\circ}\cap S^{k}\subset (\R^{k+1})^{*}$ we define
$$B=\{(\omega, x)\in \Omega \times \P^{n}\, : \, (\omega p)(x)>0\}.$$
Notice that the previous definition makes sense since for every $\omega \in \Omega$ the map $p\omega$ is homogeneous of degree \emph{two}.
The following lemma relates the topology of $B$ to that of $X.$

\begin{lemma}\label{lemmahomeq}The projection $\beta_{r}$ on the second factor defines a homotopy equivalence between $B$ and $\P^{n}\backslash X=\beta_{r}(B).$
\end{lemma}
\begin{proof}The equality $\beta_{r}(B)=\P^{n}\backslash X$ follows from $(K^{\circ})^{\circ}=K.$ For every $x\in \P^{n}$ the set $\beta_{r}^{-1}(x)$ is the intersection of the set $\Omega\times\{x\}$ with an open half space in $(\R^{k+1})^{*}\times \{x\}.$ Let $(\omega_{x},x)$ be the center of gravity of the set $\beta_{r}^{-1}(x).$ It is easy to see that $\omega_{x}$ depends continuosly on $x\in \beta_{r}(B).$ Further it follows form convexity considerations that $(\omega_{x}/\,\|\omega_{x}\|,x)\in B$ and for any $(\omega, x)\in B$ the arc $(\frac{t\omega_{x}+(1-t)\omega_{x}}{\|t\omega_{x}+(1-t)\omega_{x}\|},x),\, 0\leq t\leq1$ lies entirely in $B.$ It is clear that $x\mapsto (\omega_{x}/\,\|\omega_{x}\|,x), \, x\in \beta_{r}(B)$ is a homotopy inverse to $\beta_{r}.$
\end{proof}

We first construct a slightly more general spectral sequence $(F_{r},d_{r})$ converging to $H^{*}(\Omega\times \P^{n},B)$ which in general is \emph{not} isomorphic to $H_{n-*}(X).$ The required spectral sequence $(E_{r},d_{r})$ arises by applying the following Theorem to a modification $(\hat{q},\hat{K})$ of the pair $(q,K)$ such that $H^{*}(\hat{\Omega}\times \P^{n},\hat{B})\simeq H_{n-*}(X).$

\begin{teo} \label{basiccoo} There exists a cohomology spectral sequence of the first quadrant $(F_{r},d_{r})$ converging to $H^{*}(\Omega\times \P^{n},B;\Z_{2})$ such that for every $i,j\geq0$ $$F_{2}^{i,j}=H^{i}(\Omega, \Omega^{j+1};\Z_{2}).$$
\end{teo}

\begin{proof}
Fix a scalar product $q_{0}$ (i.e. a positive definite form) and consider the function $\alpha:\Omega\times \P^{n}\to \R$ defined by $(\omega, x)\mapsto (\omega p)(x).$ Notice that until the scalar product $q_{0}$ has not been fixed, then only the \emph{sign} of $\alpha$ is well defined (because $p$ is homogeneous of degree \emph{two}); once $q_{0}$ has been fixed we can talk also about its \emph{value} by defining it to be that of the restriction to the $q_{0}$-unit sphere; we will discuss this better later.
The function $\alpha$ is continuos semialgebraic and $B=\{\alpha >0\}.$ By semialgebraicity, there exists $\epsilon >0$ such that the inclusion $$C(\epsilon)=\{\alpha \geq \epsilon\}\hookrightarrow B$$ is a homotopy equivalence.\\
Consider the projection $\beta_{l}(\epsilon):C(\epsilon)\to \Omega$ on the first factor; then by Leray there exists a cohomology spectral sequence $(F_{r}(\eps), d_{r}(\eps))$ converging to the cohomology group $H^{*}(\Omega \times \P^{n}, C(\eps);\Z_{2})\simeq H^{*}(\Omega \times \P^{n},B;\Z_{2})$ such that $$F_{2}^{i,j}(\eps)=\check{H}^{i}(\Omega, \mathcal{F}^{j}(\eps))$$ where $\mathcal{F}^{j}(\eps)$ si the sheaf generated by the presheaf $V\mapsto H^{j}(V\times \P^{n},\beta_{l}(\eps)^{-1}(V);\Z_{2}).$ Since $C(\epsilon)$ and $\Omega$ are locally compact and $\beta_{l}(\eps)$ is proper ($C(\eps)$ is compact), then the following isomorphism holds for the stalk of $\mathcal{F}^{j}(\eps)$ at each $\omega \in \Omega$ (see  \cite{Godement}, Remark 4.17.1, p. 202): $$(\mathcal{F}^{j}(\eps))_{\omega}\simeq H^{j}(\{\omega\}\times \P^{n},\beta_{l}(\eps)^{-1}(\omega);\Z_{2}).$$
The set $\beta_{l}(\eps)^{-1}(\omega)= \{x\in \P^{n}\, : \, (\omega p)(x)\geq \epsilon\}=\{x\in \P^{n}\, : \, (\omega p-\epsilon q_{0})(x)\geq 0\}$ has the homotopy type of a projective space of dimension $n-\ind^{-}(\omega p -\eps q_{0});$ thus, if we set $\ii^{-}(\eps)$ for the function $\omega \mapsto \ind^{-}(\omega p -\eps q_{0}),$ the following holds: $$(\mathcal{F}^{j}(\eps))_{\omega}=\left\{ \begin{array}{cc} \Z_{2}&\textrm{if $\ii^{-}(\eps)(\omega)>n-j$;}\\ 0& \textrm{otherwise}\end{array}\right.$$ Thus the sheaf $\mathcal{F}^{j}(\eps)$ is zero on the closed set $\Omega_{n-j}(\eps)=\{\ii^{-}(\eps)\leq n-j\}$ and is locally constant with stalk $\Z_{2}$ on its complement; hence $$F_{2}^{i,j}(\eps)=\check{H}^{i}(\Omega, \mathcal{F}^{j}(\eps))=\check{H}^{i}(\Omega, \Omega_{n-j}(\eps);\Z_{2}).$$
We claim now that $\Omega^{j+1}=\bigcup_{\eps>0}\Omega_{n-j}(\eps)$. Let $\omega \in \bigcup_{\eps>0}\Omega_{n-j}(\eps);$ then there exists $\overline{\eps}$ such that $\omega\in \Omega_{n-j}(\eps)$ for every $\eps<\overline{\eps}.$ Since for $\eps$ small enough $$\ii^{-}(\eps)(\omega)=\ii^{-}(\omega)+\dim (\ker \omega p)$$ then it follows that $$\ii^{+}(\omega)=n+1-\ii^{-}(\omega)-\dim (\ker \omega p)\geq j+1.$$ Viceversa if $\omega \in \Omega^{j+1}$ the previous inequality proves $\omega\in \Omega_{n-j}(\eps)$ for $\eps$ small enough, i.e. $\omega \in \bigcup_{\eps>0}\Omega_{n-j}(\eps).$\\
Moreover if $\omega\in \Omega_{n-j}(\eps)$ then, eventually choosing a smaller $\eps$, we may assume $\eps$ properly separates the spectrum of $\omega p$ and thus, by algebraicity of the map $\omega\mapsto \omega p$, there exists $U$ open neighborhood of $\omega$ such that $\eps$ properly separates also the spectrum of $\omega p'$ for every $\omega'\in U;$ hence $\omega'\in \Omega_{n-j}(\eps)$ for every $\omega'\in U.$ From this consideration it easily follows that each compact set in $\Omega^{j+1}$ is contained in some $\Omega_{n-j}(\eps)$ and thus $$\varinjlim_{\eps}\{H_{*}(\Omega, \Omega_{n-j}(\eps))\}=H_{*}(\Omega, \Omega^{j+1}).$$ With this in mind the following chain of isomorphisms $\varprojlim\{H^{i}(\Omega, \Omega_{n-j}(\eps);\Z_{2})\}\simeq(\varinjlim\{H_{i}(\Omega, \Omega_{n-j}(\eps);\Z_{2})\})^{*}=(H_{i}(\Omega, \Omega^{j+1};\Z_{2}))^{*}$ gives $$F_{2}^{i,j}=H^{i}(\Omega, \Omega^{j+1};\Z_{2}).$$

\end{proof}

\begin{remark}In the case $K \ne -K$, i.e. $\Omega \ne S^l,$ then $(E_{r},d_{r})$ converges to $H_{n-*}(X, \Z_{2}).$
This follows by comparing the two cohomology long exact sequences of the pairs $(\Omega\times \P^{n},B)$ and $(\P^{n},\P^{n}\backslash X)$ via the map $\beta_{r}.$ In this case $\beta_{r}:\Omega\times \P^{n}\to \P^{n}$ is a homotopy equivalence and the Five Lemma and Lemma \ref{lemmahomeq} together give $$H^{*}(\Omega\times \P^{n},B)\simeq H^{*}(\P^{n},\P^{n}\backslash X)\simeq H_{n-*}(X)$$
the last isomorphism being given by Alexander-Pontryagin Duality.
\end{remark}

\begin{thma}

There exists a cohomology spectral sequence of the first quadrant $(E_{r},d_{r})$ converging to $H_{n-*}(X;\Z_{2})$ such that $$E_{2}^{i,j}=H^{i}(C\Omega, \Omega^{j+1};\Z_{2}).$$
\end{thma}
\begin{proof}
Keeping in mind the previous remark, we work the general case (i.e. also the case $K=\{0\}$). We replace $K$ with $\hat{K}=(-\infty, 0]\times K,$ the map $p$ with the map $\hat{p}:\R^{n+1}\to \R^{k+2}$ defined by $\hat{p}=(-q_{0}, p),$ where $q_{0}$ is a positive definite form (i.e. a scalar product)  and $\Omega$ with $$\hat{\Omega}=\hat{K}^{\circ}\cap S^{k+1}.$$ We also define $$\hat{\Omega}^{j+1}=\{(\eta,\omega)\in \hat{\Omega}\, : \, \ind^{+}(\omega p -\eta q_{0})\geq j+1\}.$$
Then, by construction, $$\hat{p}^{-1}(\hat{K})=p^{-1}(K)=X.$$
Applying Theorem \ref{basiccoo} to the pair $(\hat{p}, \hat{K}),$ with the previous remark in mind, we get a spectral sequence $(\hat{E}_{r},\hat{d}_{r})$ converging to $H_{n-*}(X;\Z_{2})$ with $$\hat{E}_{2}^{i,j}=H^{i}(\hat{\Omega}, \hat{\Omega}^{j+1};\Z_{2}).$$
We identify $\Omega^{j+1}$ with $\hat{\Omega}^{j+1}\cap\{\eta=0\}$ and we claim that the inclusion of pairs $(\hat{\Omega}, \Omega^{j+1})\hookrightarrow (\hat{\Omega}, \hat{\Omega}^{j+1})$ induces an isomorphism in cohomology.  This follows from the fact that $\hat{\Omega}^{j+1}$ deformation retracts onto $\Omega^{j+1}$ along the meridians (the deformation retraction is defined since $j\geq 0$ and $\ii^{+}(1,0,\ldots,0)=0,$ thus the ``north pole'' of $S^{k+1}$ does not belong to any of the $\hat{\Omega}^{j+1}$). If $\eta_{1}\leq \eta_{2}$ then $\ind^{+}(\omega p -\eta_{1}q_{0})\geq \ind^{+}(\omega p -\eta_{2}q_{0}):$ thus if $(\eta,\omega)\in \hat{\Omega}^{j+1}$ then all the points on the meridian arc connecting $(\eta, \omega)$ with $\Omega=\hat{\Omega}\cap \{\eta=0\}$ belong to $\hat{\Omega}^{j+1}.$\\
Noticing that $(\hat{\Omega}, \Omega^{j+1})\approx (C \Omega, \Omega^{j+1}),$ where $C \Omega$ stands for the topological space cone of $\Omega,$ concludes the proof.
\end{proof}

If we define $$\mu\doteq\max_{\eta \in\Omega}\ii^{+}(\eta),$$ then by looking directly at the table in the Introduction we can derive the following corollary of Theorem A.

\begin{coro}\label{corempty}If $0\leq b\leq n-\mu-k$ then $$H_{b}(X)=\Z_{2}.$$
In particular if $n\geq \mu +k$ then $X$ is nonempty.
\end{coro}
\begin{proof}Simply observe that the group $E_{2}^{0,n-b}$ equals $\Z_{2}$ for $0\leq b\leq n-\mu-k$ and that all the differentials $d_{r}:E_{r}^{0,n-b}\to E_{r}^{r,n-b+r-1}$ for $r\geq0$ are zero, since they take values in zero elements. Hence $$\Z_{2}=E_{\infty}^{0,n-b}=H_{b}(X).$$
\end{proof}

\section{Preliminaries for the second differential}
\subsection{Nondegeneracy properties}

Let $\mathcal{Q}_{0}\subset \mathcal{Q}$ be the set of singular quadratic forms on $\R^{n+1}$:
$$\mathcal{Q}_{0}=\{q\in \mathcal{Q}\, : \, \ker (q)\neq0\}.$$
Consider the set $K=\{(x,q)\in \R^{n+1}\times\mathcal{Q}\, | \, x\in \ker{q}\}$ and the map $p:K\to \mathcal{Q}$ which is the restriction of the projection on the second factor. Let $$\mathcal{Q}_{0}=\coprod Z_{j}$$ be a Nash stratification (i.e. smooth and semialgebraic) such that $p$ trivializes over each $Z_{j}.$\\
For a quadratic form $q\in \Q$ we may abuse a little of notations and write $q(\cdot,\cdot)$ for the bilinear form obtained by polarizing $q;$ no confusion will arise by distinguish the two from the number of their arguments.\\
We notice the following:
\begin{fact}
Let $r$ be a singular form and suppose $r\in Z_{j}$ for some stratum of $\mathcal{Q}_{0}$ as above. Then for every $q\in T_{r}Z_{j}$ and $x_{0}\in ker(r)$ we have $q(x_{0},x_{0})=0.$
\end{fact}
\begin{proof}
Let $r:I\to Z_{j}$ be a smooth curve such that $r(0)=r$ and $\dot{r}(0)=q.$ By the triviality of $p$ over $Z_{j}$ it follows that there exists $x:I\to \R^{n+1}$ such that $x(0)=x_{0}$ and $x(t)\in \ker(r(t))$ for every $t\in I. $ This implies $r(t)(x(t),x(t))\equiv 0$ and deriving we get
$$0=\dot{r}(0)(x(0),x(0))+2r(0)(x(0),\dot{x}(0))=q(x_{0},x_{0}).$$
\end{proof}

\begin{defi}
Let $f:\Omega\to \mathcal{Q}$ be a smooth map. We say that $f$ is degenerate at $\omega_{0}\in\Omega$ if there exists $x\in \ker(f(\omega_{0}))\backslash\{0\}$ such that for every $v\in T_{\omega_{0}}\Omega$ we have $(df_{\omega_{0}}v)(x,x)\leq 0;$ in the contrary case we say that $f$ is nondegenerate at $\omega_{0}.$ We say that $f$ is nondegenerate if it is nondegenerate at each point $\omega \in \Omega.$
\end{defi}

\begin{lemma}
Let $\Omega=\coprod V_{i}$ be a finite partiton with each $V_{i}$ Nash and $f:\Omega\to\mathcal{Q}$ be a semialgebraic map and $\mathcal{Q}_{0}=\coprod Z_{j}$ as above. Suppose that for every $V_{i}$ the map $f_{|V_{i}}$ is transversal to all strata of $\mathcal{Q}_{0}.$ Then $f$ is nondegenerate.
\end{lemma}
\begin{proof}
Let $\omega_{0}\in \Omega$ and $x\in \ker(f(\omega_{0}))\backslash\{0\};$ we must prove that there exists $v\in T_{\omega_{0}}\Omega$ such that $(df_{\omega_{0}}v)(x,x)>0.$ Let $V_{i}$ such that $\omega_{0}\in V_{i}.$ Then $T_{\omega_{0}}V_{i}\subset T_{\omega_{0}}\Omega;$ suppose $f(\omega_{0})\in Z_{j}.$ Since $f_{|V_{i}}$ is transversal to $Z_{j},$ then
$$\textrm{im} (df_{|V_{i}})_{\omega_{0}}+T_{f(\omega_{0})}Z_{j}=\mathcal{Q}.$$
Thus let $q^{+}\in \mathcal{Q}$ be a positive definite form, $v\in T_{\omega_{0}}V_{i}$ and $\dot{r}\in T_{f(\omega_{0})}Z_{j}$ such that
$$df_{\omega_{0}}v+\dot{r}=q^{+}.$$
Since $x\in \ker (f(\omega_{0}))\backslash\{0\},$ then the previous Fact implies $\dot{r}(x,x)=0,$ and plugging in the previous equation we get
$$(df_{\omega_{0}}v)(x,x)=(df_{\omega_{0}}v)(x,x)+\dot{r}(x,x)=q^{+}(x,x)>0.$$

\end{proof}

\begin{lemma}\label{lemmatrasv}
Let $f:\Omega \to \mathcal{Q}$ be a semialgebraic smooth map. Then there exists a definite positive form $q_{0}\in \mathcal{Q}$ such that for every $\eps>0$ sufficiently small the map $f_{\eps}:\Omega \to \mathcal{Q}$ defined by $$\omega \mapsto f(\omega)-\eps q_{0}$$ is nondegenerate.
\end{lemma}
\begin{proof}
Let $\Omega=\coprod V_{i}$ and $\mathcal{Q}_{0}=\coprod Z_{j}$ be as above. For every $V_{i}$ consider the map $F_{i}:V_{i}\times \mathcal{Q}^{+}\to \mathcal{Q}$ defined by
$$(\omega,q_{0})\mapsto f(\omega)-q_{0}.$$
Since $\mathcal{Q}^{+}$ is open in $\mathcal{Q},$ then $F_{i}$ is a submersion and $F_{i}^{-1}(\mathcal{Q}_{0})$ is Nash-stratified by $\coprod F_{i}^{-1}(Z_{j}).$ Then ${(F_{q_{0}})}_{|V_{i}}:\omega \mapsto f(\omega)-q_{0}$ is transversal to all strata of $\mathcal{Q}_{0}$ if and only if $q_{0}$ is a regular value for the restriction of the second factor projection $\pi_{i}:V_{i}\times \mathcal{Q}^{+}\to \mathcal{Q}^{+}$ to each stratum of $F_{i}^{-1}(\mathcal{Q}_{0})=\coprod F_{i}^{-1}(Z_{j}).$
Thus let $\pi_{ij}=(\pi_{i})_{|F_{i}^{-1}(Z_{j})}:F_{i}^{-1}(Z_{j})\to \mathcal{Q}^{+};$ since all datas are smooth semialgebraic, then by semialgebraic Sard's Lemma, the set $\Sigma_{ij}=\{\hat{q}\in \mathcal{Q}^{+}\, : \, \hat{q}\textrm{ is a critical value of $\pi_{ij}$}\}$ is a semialgebraic subset of $\mathcal{Q}^{+}$ of dimension $\dim (\Sigma_{ij})<\dim (\mathcal{Q}^{+}).$ Hence $\Sigma=\cup_{i,j}\Sigma_{ij}$ also is a semialgebraic subset of $\mathcal{Q}^{+}$ of dimension $\dim (\Sigma)<\dim (\mathcal{Q}^{+})$ and for every $q_{0}\in \mathcal{Q}^{+}\backslash \Sigma$ and for every $i,j$ the restriction of $\omega\mapsto f(\omega)-q_{0}$ to $V_{i}$ is transversal to $Z_{j}.$ Thus by the previous Lemma $f-q_{0}$ is nondegenerate. Since $\Sigma$ is semialgebraic of codimension at least one, then there exists $q_{0}\in \mathcal{Q}^{+}\backslash \Sigma$ such that $\{t q_{0}\}_{t>0}$ intersects $\Sigma$ in a finite number of points, i.e. for every $\eps>0$ sufficiently small $\eps q_{0}\in \mathcal{Q}^{+}\backslash \Sigma.$ The conclusion follows.
\end{proof}

Let $f:\Omega \to \mathcal{Q}$ be a smooth map. We define, for every $V\subset \Omega$ the set $$B_{f}(V)=\{(\omega, x)\in V\times \P^{n}\, : \, f(\omega)(x)>0\}.$$
Notice that for the previous definition the \emph{value} of $f(\omega)$ at $x\in \P^{n},$ which is still undefined, is irrelevant: what we need, i.e. its sign, is well defined since $f(\omega)$ is a quadratic form, hence homogeneous of degree \emph{two}.\\
Suppose now that a scalar product in $\R^{n+1}$ has been fixed. Then we can identify each $q\in \Q$ with a symmetric $(n+1)\times (n+1)$ matrix $Q$ by the rule:
$$q(x)=\langle x,Qx\rangle .$$
Now also the \emph{value} of $q$ at $x\in \P^{n}$ is defined: let $S^{n}$ be the unit sphere (w.r.t. the fixed scalar product) in $\R^{n+1}$ and $p:S^{n}\to \P^{n}$ be the covering map; then, with a little abuse of notations, we set for $x=p(v)\in \P^{n}$ (for some $v\in S^{n}):$ $$q(x)\doteq q(v).$$
Since $q$ is homogeneous of even degree, the previous function is well defined, i.e. does not depend on the choice of $v.$\\
The $\emph{eigenvalues}$ of $q$ with respect to $g$ are defined to be those of $Q$:
$$\lambda_{1}(q)\geq\cdots\geq \lambda_{n+1}(q).$$
In the space $\Q$ we define
$$ \D_{j}\doteq \{q\in \Q\, : \, \lambda_{j}(q)\neq \lambda_{j+1}(q)\}.$$
Notice that  $\Q^{j}\backslash \Q^{j+1}\subset \D_{j}$ for every possible choice of the scalar product in $\R^{n+1}.$ On the space $\D_{j}$ is naturally defined the vector bundle:
$$\begin{tikzpicture}[xscale=1.7, yscale=1.5]

     \node (A4_0) at (4, 0) {$\D_{j}$};
    \node (A3_1) at (3, 1) {$\R^{j}$};
    \node (A4_1) at (4, 1) {$\L_{j}^{+}$};
\path (A3_1) edge [->] node [auto] {}(A4_1);
    \path (A4_1) edge [->] node [auto] {}(A4_0);

      \end{tikzpicture}
$$
whose fiber at the point $q\in \D_{j}$ is $(\L_{j}^{+})_q=\textrm{span}\{x\in \R^{n+1}\, : \, Qx=\lambda_{i}x, 1\leq i\leq j\}$ and whose vector bundle structure is given by its inclusion in $\D_{j}\times \R^{n+1}.$\\
Similarly the vector bundle $\R^{n-j+1}\hookrightarrow \mathcal{L}_{j}^{-}\to \D_{j}$ has fiber at the point $q\in \D_{j}$ the vector space $(\L_{j}^-)_q=\textrm{span}\{x\in \R^{n+1}\, : \, Qx=\lambda_{i}x, j+1\leq i\leq n+1\}$ and vector bundle structure given by its inclusion in $\D_{j}\times \R^{n+1}.$
Notice that $$\mathcal{L}_{j}^+\oplus \mathcal{L}_{j}^-=\D_{j}\times \R^{n+1}$$ and thus Whitney product formula holds for their total Stiefel-Whitney classes:
$w(\mathcal{L}_{j}^{+})\smile w(\mathcal{L}_{j}^-)=1.$ In particular:
$$w_{1}(\mathcal{L}_{j}^+)=w_{1}(\mathcal{L}_{j}^-).$$
In the sequel we will need for $q\in \D_{j}$ the projective spaces:
$$P_{j}^{+}(q)\doteq\mathbb{P}(\L_{j}^+)_q \quad \textrm{and} \quad P_{j}^{-}(q)\doteq \P(\L_{j}^{-})_{q}.$$
For a given $q\in \Q$ with $\ii^{-}(q)=i$ (which implies $q\in D_{n+1-i}$) we will use the simplified notation $$P^{+}(q)\doteq P_{n+1-i}^{+}(q)\quad \textrm{and}\quad P^{-}(q)\doteq P_{n+1-i}^{-}(q).$$
(even if $q\in \D_{n+1-i}$ for every metric still there is dependence on the metric for these spaces, but we omit it for brevity of notations; the reader should pay attention).
Notice that $q|_{P^{-}(q)}<0$ whereas $q|_{P^{+}(q)}\geq0,$ i.e. $P^{+}(q)$ contains also $\mathbb{P}(\ker q).$
The following picture may help the reader:
$$\underbrace{\lambda_{1}(q)\geq\cdots\geq\lambda_{n+1-\ii^{-}(q)}(q)}_{P^{+}(q)}\geq0>\underbrace{\lambda_{n+2-\ii^{-}(q)}(q)\geq \cdots \geq \lambda_{n+1}(q)}_{P^{-}(q)}$$

\begin{lemma}\label{lemmadef} Let $f:\Omega\to \mathcal{Q}$ be a smooth \emph{nondegenerate} map. Then there exists $\delta_{1}:\Omega \to (0,+\infty)$ such that for every $\omega\in \Omega,$ for every $V_{1}\subset V_{2}$ closed convex neighborhoods of $\omega$ with $\textrm{diam}(V_{2})<\delta_{1}(\omega)$ and for every $\eta \in V_{1}$ such that $\ii^{-}(f(\eta))=\ii^{-}(f(\omega))$ and $\det (f(\eta))\neq 0$ the inclusions $$(\eta, P^{+}(f(\eta)))\hookrightarrow B_{f}(V_{1})\hookrightarrow B_{f}(V_{2})$$ are homotopy equivalences.\\
Moreover in the case $f$ is semialgebraic, then the function $\delta_{1}$ can be chosen to be semialgebraic (but in general not continuous).
\end{lemma}
\begin{proof} The existence of $\delta_{1}$ is a direct consequence of Lemma 8 of \cite{Agrachev2}. The fact that $\delta_{1}$ can be chosen to be semialgebraic if $f$ is semialgebraic follows directly from the proof of Lemma 7 of \cite{Agrachev2}.
\end{proof}

\subsection{Negativity properties}

Let now $f:\Omega\to \Q$ and $\omega\in \Omega;$ let $M(\omega)<0$ be such that $$\lambda_{n+2-\ii^{-}(f(\omega))}(f(\omega))<M(\omega)$$ (notice that by definition $\lambda_{n+2-\ii^{-}(\omega)}(f(\omega))$ is the biggest negative eigenvalue of $f(\omega);$ see the above diagram for the numbering of the eigenvalues of a quadratic form).
Then by continuity there exists $\delta_{2}''(\omega)$ such that for every neighborhood $V$ of $\omega$ with $\textrm{diam}(V)<d_{2}''(\omega)$ and for every $\eta \in V$ $$\lambda_{n+2-\ii^{-}(f(\omega))}(f(\eta))<M(\omega).$$
Thus for every neighborhood $U$ of $\omega$ with $\textrm{diam}(U)<\delta_{2}''(\omega)$ we define: $$P^{-}(\omega, U)=\{x \in \P^{n}\, : \, \exists \eta \in U \,\textrm{s.t.}\, x\in P_{n+1-\ii^{-}(f(\omega))}^{-}(f(\eta))\}.$$
We claim the following.
\begin{lemma}For every $\omega \in \Omega$ there exists $0<\delta_{2}'(\omega)<\delta_{2}''(\omega)$ such that for every neighborhood of $\omega$ with $\textrm{diam}(V)<\delta_{2}'(\omega)$ $$\textrm{\emph{Cl}}(P^{-}(\omega, V))\subseteq \P^{n}\backslash \{f(\omega)(x)\geq 0\}.$$
\end{lemma}
\begin{proof} By absurd suppose for every $k\in \mathbb{N}$ the two sets $\textrm{Cl}({P^{-}(\omega, B(\omega, 1/k))})$ and $\{f(\omega)(x)\geq 0\}$ intersect. Then for every $k\in \N$ there exists a sequence $x_{k}^{l}\to x_{k}$ such that for every $x_{k}^{l}$ there exists $\omega_{k}^{l}\in B(\omega, 1/k)$ such that $x_{k}^{l}\in P^{-}_{n+1-\ii^{-}(\omega)}(f(\omega_{k}^{l}))$ and $f(\omega)(x_{k})\geq 0.$\\
Then it follows that $f(\omega_{k}^{l})(x_{k}^{l})<M(\omega)$ (recall that the function $(\omega, x)\mapsto f(\omega)(x)$ is defined, once the scalar product has been fixed, to be the restriction of $f(\omega)$ to the unit sphere covering $\P^{n}$) and, by extracting convergent subsequences, that  $$0\leq \lim_{k\to \infty}f(\omega)(x_{k})=\lim_{k\to \infty}f(\omega_{k})(x_{k})\leq M(\omega)$$ which is absurd since $M(\omega)<0$ by definition.
\end{proof}

Notice that in the case $f$ is semialgebraic then $\omega\mapsto M(\omega)$ can be chosen semialgebraic and hence $\omega\mapsto \delta_{2}'(\omega)$ also can be chosen to be semialgebraic.

\begin{lemma}\label{lemmaneg} For every $\omega\in \Omega$ there exists $0<\delta_{2}(\omega)<\delta_{2}''(\omega)$ such that for every neighborhood $V$ of $\omega$ with $\textrm{diam}(V)<\delta_{2}(\omega)$ the following holds: $$\emph{Cl}({P^{-}(\omega, V)})\subset \P^{n}\backslash \beta_{r}(B_{f}(V)).$$
Moreover in the case $f$ is semialgebraic, then $\omega\mapsto \delta_{2}(\omega)$ can be chosen semialgebraic.
\end{lemma}
\begin{proof}
Let $W$ be a neighborhood of $\omega$ with $\textrm{diam}(W)<\delta_{2}'(\omega).$ Then the two compact sets $\textrm{Cl}({P^{-}(\omega, W)})$ and $\{f(\omega)(x)\geq 0\}$ do not intersect by the previous Lemma. Consider the continuous function $a:\textrm{Cl}({W})\times \P^{n}\to \R$ defined by $a(\eta, x)=f(\eta)(x)$ and a neighborhood $U$ of $\{f(\omega)(x)\geq 0\}$ in $\P^{n}$ disjoint form $\textrm{Cl}({P^{-}(\omega, W)}).$ Then $\beta_{r}^{-1}(U)\cap \{a\geq 0\}$ is an open neighborhood of $\{\omega\}\times \{f(\omega)(x)\geq 0\}$ in $\{a \geq 0\}.$ Consider now $b:\{a\geq 0\}\to \R$ defined by $(\eta,x)\mapsto d(\eta, \omega).$ Then, since $\{a\geq 0\}$ is compact, the family $\{b^{-1}[0,\delta)\}_{\delta>0}$ is a fundamental system of neighborhoods of $b^{-1}(0)=\{\omega\}\times \{f(\omega)(x)\geq 0\}$ in $\{a\geq 0\}.$ Thus there exists $\overline{\delta}$ such that $b^{-1}[0,\overline{\delta})\subset \beta_{r}^{-1}(U)\cap \{a\geq 0\}.$ Hence any $\delta_{2}(\omega)$ such that $B(\omega, 3\delta_{2}(\omega))\subset B(\omega, \overline{\delta})\cap W$ satisfies the requirement, since every neighborhood $V$ of $\omega$ with $\textrm{diam}(V)<\delta_{2}(\omega)$ is contained in $B(\omega, 3\delta_{2}(\omega))$ and
\begin{align*}\textrm{Cl}({P^{-}(\omega, B(\omega, 3\delta_{2}(\omega))})&\subset \textrm{Cl}({P^{-}(\omega, W)})\\ &\subset \P^{n}\backslash\beta_{r}( \{a\geq 0\})\subset \P^{n}\backslash \beta_{r}(B_{f}(B(\omega, 3\delta_{2}(\omega)))).\end{align*}
It is clear from the construction that in the case $f$ is semialgebraic the function $\omega\mapsto \delta_{2}(\omega)$ can be chosen semialgebraic too.
\end{proof}

\subsection{Convexity properties}

We discuss here some useful facts related to convex open sets of $\R^{k}.$ We begin with the following; recall that for a given convex function $a$ and $c\in \R$ the set $\{a<c\}$ is convex.

\begin{lemma}\label{conv1}Let $a:\R^{n}\to [0, \infty)$ be a proper convex function of class $C^{2}$, $x_{0}\in \R^{n}$ such that $da_{x_{0}}\equiv 0$ and the Hessian $\textrm{He}(a)_{x_{0}}$ of $a$ at $x_{0}$ is positive definite. Let also $\psi:\R^{n}\to \R^{n}$ be a diffeomorphism. Then there exists $\overline{\epsilon}>0$ such that for every $\epsilon<\overline{\eps}$
$$\psi(\{a< \eps\})\quad\textrm{is convex}.$$
\end{lemma}
\begin{proof} Let $\phi$ be the inverse of $\psi,$ $y_{0}=\psi(x_{0})$ and $\hat{a}\doteq a \circ \phi.$ Then the set $\psi(\{a<\eps\})$ equals $\{\hat{a}<\eps\}.$ Since $da_{x_{0}}\equiv 0,$ then
$$\textrm{He}(\hat{a})_{y_{0}}={}^{t}J\phi_{y_{0}}\textrm{He}(a)_{x_{0}}J\phi_{y_{0}}>0$$ and thus, by continuity of the map $y\mapsto \textrm{He}(\hat{a})_{y},$ the function $\hat{a}$ is convex on $B(y_{0}, \eps')$ for sufficiently small $\eps';$ hence for every $c>0$ the set $\{\hat{a}_{|B(y_{0},\eps')}<c\}$ is convex.
Since $a$ is proper, then there exists $\eps$ such that $\{y\, : \, a(\phi(y))<\eps\}\subset B(y_{0},\eps').$ Thus $\{\hat{a}<\eps\}=\{\hat{a}_{|B(y_{0},\eps')}<\eps\}$ is convex.\end{proof}

Consider a family of functions $a_{w}:x\mapsto a(x+x_{0}-w), w\in W\subset \R^{n}$ with compact closure, with $a$ satisfying the conditions of the previous lemma. Since $\textrm{He}(a_{w})_{x}=\textrm{He}(a)_{x},$ then the exstimate on $\textrm{He}(a_{w})_{w}$ can be made uniform on $W.$ In particular taking $a(x)=|x|^{2}$ we derive the following corollary.

\begin{coro}
Let $U$ be an open subset of $\R^{n}$ and $\psi:U\to \R^{n}$ be a diffeomorphism onto its image. Then for every $x\in U$ there exists $\delta_{c}(x)>0$ such that for every $B(y,r)\subset B(x, 3\delta_{c}(x))$ with $r<\delta_{c}(x)$ then $$\psi(B(y,r))\quad \textrm{is convex}.$$
\end{coro}

\subsection{Construction of regular covers}
We recall the following useful result describing the local topology of the space of quadratic forms.

\begin{propo}\label{topquad}Let $q_{0}\in \mathcal{Q}$ be a quadratic map and let $V$ be its kernel. Then there exists a neighborhood $U_{q_{0}}$ of $q_{0}$ and a smooth semialgebraic map $\phi:U_{q_{0}}\to \Q(V)$ such that: 1) $\phi(q_{0})=0$; 2) $\ii^{-}(q)=\ii^{-}(q_{0})+\ii^{-}(\phi(q));$ 3) $\dim\ker (q)=\dim\ker(\phi(q));$ 4) $d\phi_{q_{0}}(p)=p_{|V}.$
\end{propo}
For the proof of the previous we refer the reader to \cite{Agrachev2}.

\begin{lemma}
Let $f:\Omega \to \mathcal{Q}$ be a smooth function transversal to all strata of $\mathcal{Q}_{0}=\coprod Z_{j}.$ Then for every $\omega\in \Omega$ there exists $\delta_{3}'(\omega)>0$ and $\psi:B(\omega, \delta_{3}'(\omega))\to \Q(\ker f(\omega))\times \R^{l}$ a diffeomorphism onto its image such that
$$\begin{tikzpicture}[xscale=1.7, yscale=1.7]

    \node (A2_0) at (2, 0) {$\Q(\ker f(\omega))$};
    \node (A1_1) at (1, 1) {$B(\omega, \delta_{3}'(\omega))$};
    \node (A3_1) at (3, 1) {$\Q(\ker f(\omega))\times \R^{l}$};

    \path (A1_1) edge [->] node [auto] {$\psi$} (A3_1);
    \path (A1_1) edge [->] node [auto,swap] {$\phi \circ f$} (A2_0);
    \path (A3_1) edge [->] node [auto] {$p_{1}$} (A2_0);

      \end{tikzpicture}
$$
is commutative. Moreover in the case $f$ is semialgebraic then $\omega\mapsto \delta_{3}'(\omega)$ can be chosen to be semialgebraic too.
\end{lemma}

\begin{proof}
If $\det (f(\omega))\neq 0$ then let $\delta_{3}'(\omega)>0$ be such that $f(B(\omega, \delta_{3}'(\omega)))\cap \mathcal{Q}_{0}=\emptyset;$ in the contrary case let $f(\omega)\in Z_{j}$ for some $j.$
Consider $\phi:U_{f(\omega)}\to \Q(\ker f(\omega))$ the map given by the previous lemma. Since $d\phi_{f(\omega)}p=p_{|\ker f(\omega)}$ then $d\phi_{f(\omega)}$ is surjective. On the other hand by transversality of $f$ to $Z_{j}$ we have:
$$ \textrm{im} (df_{\omega})+T_{f(\omega)}Z_{j}=\mathcal{Q}$$
Since $\phi(Z_{j})=\{0\},$ which implies $(d\phi_{f(\omega)})|_{T_{f(\omega)}Z_{j}}=0,$ then
$$\Q(\ker f(\omega))=\textrm{im}(d\phi_{f(\omega)})=\textrm{im}(d(\phi\circ f)_{\omega})$$
which tells $\phi \circ f$ is a submersion at $\omega.$ Thus by the rank theorem there exists $U_{\omega}$ and a diffeomorphism onto its image $\psi:U_{\omega}\to \Q(\ker f(\omega))\times \R^{l}$ such that $p_{1}\circ \psi=\phi\circ f.$ Taking $\delta_{3}'(\omega)>0$ such that $B(\omega, \delta_{3}'(\omega))\subset U_{\omega}$ concludes the proof.\\
In the case $f$ is semialgebraic, then it is clear by construction and semialgebraic rank theorem that $\delta_{3}'$ can be chosen semialgebraic too.
\end{proof}

\begin{coro}\label{ac}
For every $\omega\in \Omega$ there exists $\delta_{3}(\omega)>0$ such that for every $B(\omega',r)\subset B(\omega, 3\delta_{3}(\omega))$ with $r<\delta_{3}(\omega)$ then
$$\psi(B(\omega',r))\quad \textrm{is convex}.$$
In particular if $\omega\in B(\omega_{k},r_{k})$ for some $\omega_{0},\ldots,\omega_{i}\in \Omega$ and $r_{0},\ldots,r_{i}<\delta_{3}(\omega),$ then for every $j\in \N$ the space
$$\{\eta\in \Omega\, : \, \ii^{-}(f(\eta))\leq n-j\}\cap (\bigcap_{k=0}^{i}B(\omega_{k},r_{k}))\quad \textrm{is acyclic}.$$
\end{coro}
\begin{proof}
The first part of the statement follows by applying the previous lemmas to $\psi:U_{\omega}\to \Q(\ker f(\omega))\times \R^{l}.$\\
For the second part notice that by Proposition \ref{topquad} we have for every $\eta\in U_{\omega}$ (using the above notations):
$$\ii^{-}(f(\eta))=\ii^{-}(f(\omega))+\ii^{-}(p_{1}(\psi(\eta))).$$
This implies that, setting as above $\Omega_{n-j}(f)\doteq\{\eta\in \Omega\, : \, \ii^{-}(f(\eta))\leq n-j\},$
$$\psi(U_{\omega}\cap\Omega_{n-j}(f))\subseteq\Q_{n-j}(\ker f(\omega))\times \R^{l},$$ where $\Q_{n-j}(\ker (f(\omega)))=\{q\in \Q(\ker f(\omega)) \, :\, \ii^{-}(q)\leq n-j\}.$
Since for each $k=0,\ldots,i$ the set $\psi(B(\omega_{k},r_{k}))$ is convex, then
$$\bigcap_{k=0}^{i}\psi(B(\omega_{k},r_{k}))\quad \textrm{is convex}$$
and by hypothesis it contains $\psi(\omega).$ Since $\Q_{n-j}(\ker f(\omega))\times \R^{l}$ (if nonempty) has linear conical structure with respect to $\psi(\omega),$ then
$$\psi(\Omega_{n-j}(f))\cap  \bigcap_{k=0}^{i}\psi(B(\omega_{k},r_{k}))\quad \textrm{is acyclic}$$
and since $\psi:\bigcap_{k}B(\omega_{k},r_{k})\subset U_{\omega} \to\Q(\ker f(\omega))\times \R^{l}$ is a homeomorphism onto its image the conclusion follows.

\end{proof}
Let now $f:\Omega \to \mathcal{Q}$ be smooth, semialgebraic and transversal to all strata of $\mathcal{Q}_{0}=\coprod Z_{j}.$ Then we define $\delta:\Omega \to (0, \infty)$ by $$\delta(\omega)=\min \{\delta_{1}(\omega), \delta_{2}(\omega), \delta_{3}(\omega)\}.$$
By construction $\delta$ can be chosen to be semialgebraic. Under this assumption we prove the following.

\begin{lemma}\label{lemmacover}Let $\mathfrak{W}$ be an open cover of $\Omega$ and $f$ and $\delta$ as above. Then there exists a locally finite refinement $\mathfrak{U}=\{V_{\alpha}=B(x_{\alpha},\delta_{\alpha}), x_{\alpha}\in \Omega\}_{\alpha \in A}$ such that for every multi-index $\bar\alpha=(\alpha_{0}\cdots \alpha_{i})$ such that  $V_{\bar\alpha}\neq \emptyset$ there exists $\omega_{\bar \alpha}\in V_{\bar\alpha}$ such that for every $k=0,\ldots, i$ the following holds:
$$B(x_{\alpha_k},\delta_{\alpha_k})\subset B(\omega_{\bar\alpha},\delta(\omega_{\bar\alpha})).$$
Moreover if for every $\bar\alpha$ multi-index we let $n_{\bar\alpha}$ be the minimum of $\ii^{-}\circ f$ over $V_{\bar\alpha}\neq \emptyset,$ then the cover $\mathfrak{U}$ can be chosen as to satisfy $$n_{\alpha_{0}\cdots \alpha_{i}}=\max\{n_{\alpha_{0}},\ldots  , n_{\alpha_{i}}\}.$$

\end{lemma}
\begin{proof}
We first set some notations. Let $\mathcal{N}=\coprod_{i=1}^{l}N_{i}\subset \Omega$ be a finite family of disjoint smooth submanifold such that $\delta_{|\mathcal{N}}$ is continuous. For $i=1,\ldots,l$ let also $N_{i}'\subset N_{i}$ be a compact subset and define $\mathcal{N}'=\coprod N_{i}'.$\\
Then there exists $\eps(\mathcal{N},\mathcal{N}')>0$ such that for $i\neq j$ the two sets $\{x\in \Omega\, : \, d(x,N_{i}')<\eps(\mathcal{N},\mathcal{N}')\}$ and $\{x\in \Omega\, : \, d(x,N_{j}')<\eps(\mathcal{N},\mathcal{N}')\}$ are disjoint.\\
Let $\mathfrak{W}_{\mathcal{N'}}$ be the cover $\{W\cap \mathcal{N}'\, : \, W\in \mathfrak{W}\}$ and $\lambda_{\mathcal{N}'}>0$ be its Lebesgue number.\\
Finally let ${\delta'}_{\mathcal{N}'}=\min_{\eta \in \mathcal{N}'}3\delta(\eta)>0$ which exists since $\delta_{|\mathcal{N}}$ is continuos and $\mathcal{N}'$ is compact.\\
We define $\delta(\mathcal{N},\mathcal{N}')>0$ to be any number such that
$$\delta(\mathcal{N},\mathcal{N}')<\min\{\eps(\mathcal{N},\mathcal{N}'),\lambda_{\mathcal{N}'}, {\delta'}_{\mathcal{N}'}\}.$$
We construct now the desired cover. Let $h:|K|\to \Omega$ be a nash semialgebraic triangulation of $\Omega$ respecting $\ii^{-}\circ f$ and such that $\delta$ is continuous on each simplex. Thus $\Omega=\coprod S_{i},$ where $i=0,\ldots,k$ and $S_{i}$ is the image under $h$ of the $i$-th skeleton  of the complex $K.$\\
Let $S_{0}=\{x_{0},\ldots,x_{v}\}$ and define $$\mathfrak{U}_{0}\doteq\{B(x_{i},\delta(S_{0},S_{0})), i=0,\ldots,v\}$$ and $T_{0}=\cup_{i}B(x_{i},\delta(S_{0},S_{0})).$\\
Now proceed inductively: first set $S_{i}=\coprod_{\sigma_{i,j}\in K_{i}}h(\sigma_{i,j})$ and $S_{i}'=\coprod h(\sigma_{i,j})\backslash T_{i-1}.$ Then let $\mathfrak{U}_{i}=\{B(x_{i}^{j},\delta_{i})\, : \, x_{i}^{j}\in S_{i}'\, \textrm{and} \,\delta_{i}<\delta(S_{i},S_{i}')\}$ be such that $\mathfrak{U}_{i}$ and $\mathfrak{U}_{i}\cap S_{i}'$ have the same combinatorics; let also $T_{i}$ be defined by
$$T_{i}=\cup_{V\in\mathfrak{U}_{i}}V.$$
With the previous settings we finally define
  $$\mathfrak{U}\doteq \mathfrak{U}_{0}\cup\cdots\cup \mathfrak{U}_{k}.$$
Then $\mathfrak{U}$ verifies by construction the requirements and this concludes the proof.
\end{proof}

Given $f$ and $\delta$ as above, then a cover $\mathfrak{U}$ satisfying the conditions of the previous lemma will be called a $f$-regular cover.

\section{The second differential}

Suppose that a scalar product on $\R^{n+1}$ has been fixed and let $w_{1}(\L_{j}^{+})\in H^{1}(\D_{j})$ be the first Stiefel-Whitney class of $\L^{+}_{j}\to \D_{j}$ (the definition of the previous bundle depends on the fixed scalar product).\\
With the previous notations we prove the following theorem which describes the second differential for the spectral sequence of Theorem \ref{basiccoo}.

\begin{teo}\label{basicdiff} Let $\partial^{*}:H^{1}(D_{j})\to H^{2}(\Omega, D_{j})$ be the connecting homomorphism. Then for every $i,j\geq 0$ the differential $d_{2}:F_{2}^{i,j}\to F_{2}^{i+2,j-1}$is given by: $$d_{2}(x)=(x\smile \partial^{*}\bar{p}^{*}w_{1}(\L_{j}^{+}))|_{(\Omega, \Omega^{j})}.$$
\end{teo}
\begin{proof}
First notice that given $x\in H^{i}(\Omega, \Omega^{j+1})$ then the product $x\smile \partial^{*}\bar{p}^{*}w_{1}(\L_{j}^{+})\in H^{i+2}(\Omega, \Omega^{j+1}\cup D_{j})$ and since $$\Omega^{j}\subset \Omega^{j+1}\cup D_{j}$$ we can consider the restriction $(x\smile\partial^{*}\bar{p}^{*}w_{1}(\L_{j}^{+}))|_{(\Omega, \Omega^{j})}\in H^{i+2}(\Omega, \Omega^{j}).$\\
We construct $(F_{r},d_{r})$ in a slightly different way than in Theorem \ref{basiccoo}, more practical for computations.\\
Let's start with a fixed scalar product $g.$ For this proof we will use in the notations for the various objects their dependence on $g.$ \\
By Lemma \ref{lemmatrasv} there exists $q_{0}>0$ such that for $\eps>0$ sufficiently small the map $f_{\eps}:\Omega\to \mathcal{Q}$ defined by
$$\omega \mapsto \omega p -\eps q_{0}$$
is nondegenerate (and also can be made transversal to $\mathcal{Q}\backslash \mathcal{D}^{g},$ where $\mathcal{D}^{g}=\cup_{j}\mathcal{D}_{j}^{g}$). Let $a:\Omega \times \P^{n}\to \R$ be the semialgebraic function defined by $(\omega, x)\mapsto (\omega p)(x)/q_{0}(x)$ (recall that we need to fix a scalar product for the definition of $a$). Then $B=\{a>0\}$ and by semialgebraicity for every $\eps>0$ sufficiently small the inclusion $$B(\eps)=\{a>\eps\}\hookrightarrow B$$ is a homotopy equivalence.\\
The proof will develop along the following idea: first we study the Leray spectral sequence $(F_{r}(\eps, \mathfrak{U}), d_{r}(\eps, \mathfrak{U}))$ of the map $(\beta_{l})_{|B(\eps)}$ with respect to $f_{\eps}$-regular cover $\mathfrak{U};$ then we perform the direct limit over the cover to get the pure Leray spectral sequence $(F_{r}(\eps), d_{r}(\eps))$ of the map $(\beta_{l})_{|B(\eps)};$ finally we perform the $\eps$-limit getting the desired spectral sequence $(F_{r},d_{r}).$\\
Thus for every $\eps>0$ let $\mathfrak{U}=\{V_{\alpha}\}_{\alpha \in A}$ be a cover of $\Omega$ regular with respect to $f_{\eps}$ and $(F_{r}(\eps, \mathfrak{U}), d_{r}(\eps, \mathfrak{U}))$ be the relative Leray spectral sequence of $(\beta_{l})_{|B(\eps)}$ with respect to the cover $\mathfrak{U}:$
$$F_{0}^{i,j}(\eps, \mathfrak{U})=\prod_{\alpha_{0}<\cdots<\alpha_{i}}C^{j}(\beta_{l}^{-1}V_{\alpha_{0}\cdots \alpha_{i}}, \beta_{l}^{-1}V_{\alpha_{0}\cdots \alpha_{i}}\cap B(\eps); \Z_{2}) $$
Let also $K_{0}^{*,*}=K_{0}^{*,*}(\mathfrak{U})$ be the Kunneth bicomplex associated to the map $\beta_{l}:\Omega\times \P^{n}\to \Omega$ with respect to $\mathfrak{U}.$ Notice that $F_{0}^{*,*}(\eps, \mathfrak{U})$ is a subcomplex of $K_{0}^{*,*}$ and we denote by $\delta_{F}, d_{F}$ and $\delta_{K}, d_{K}$ the respective bicomplex differentials (the first two are the restriction to $F_{0}^{*,*}$ of the second two).\\
For every $\omega \in \Omega$ and $\eps>0$ we let $\ii^{-}(\eps)(\omega)=\textrm{ind}^{-}(\omega p -\eps q_{0})$ and for every multi-index $\bar\alpha=(\alpha_{0},\ldots, \alpha_{i})$ such that $V_{\bar\alpha}\neq \emptyset$ we let $n_{\bar\alpha}$ be the minimum of $\ii^{-}(\eps)$ over $V_{\bar\alpha}.$ We take an order on the index set $A$ such that $$\alpha\leq \beta \Rightarrow n_{\alpha}\leq n_{\beta}.$$ In this way, by Lemma \ref{lemmacover}, for every multi-index $\bar\alpha=(\alpha_{0},\ldots, \alpha_{i})$ such that $V_{\bar\alpha}\neq \emptyset$ we have that  $n_{\bar\alpha}=n_{\alpha_{i}}.$
For every multi-index $\bar\alpha$ such that $V_{\bar \alpha}\neq \emptyset$ let $\omega_{\bar\alpha}$ be given by Lemma \ref{lemmacover}, $\ii^{-}(\eps)(\omega_{\bar\alpha})=n_{\bar\alpha},$ and we let $\eta_{\bar\alpha}\in V_{\bar\alpha}$ be such that $\det(f_{\eps}(\eta_{\bar\alpha}))\neq 0,\, \ii^{-}(\eps)(\eta_{\bar\alpha})=n_{\bar\alpha}$ and $f_{\eps}(\eta_{\bar\alpha})\in \mathcal{D}^{g}$ (such $\eta_{\bar\alpha}$ always exists, and by transversality of the map $f_{\eps}$ to $\mathcal{Q}_{0}$ and to $\mathcal{Q}\backslash \mathcal{D}^{g}$, which have respectively codimension one and two, there are plenty of them).\\
For every $0\leq j\leq n$ and $\alpha \in A$ we define
$$N(\alpha, j)=(P_{j}^{-})^{g}(f_{\eps}(\eta_{\alpha}))$$
where the $g$ on $(P_{j}^{-})^{g}$ denotes the dependence on the fixed scalar product.
Moreover we let $\nu(\alpha, j)\in C^{j}(\P^{n})$ be the Poincaré dual of $N(\alpha, j)$ and we define a cochain $\psi^{0,j}\in K_{0}^{0,j}$ by
$$\psi^{0,j}(\alpha)=\beta_{r}^{*}\nu(\alpha, j).$$
Notice that if $n-n_{\alpha}+1\leq j\leq n$ then , by Lemma \ref{lemmaneg}, $N(\alpha, j)\subset \P^{n}\backslash \beta_{r}(B_{\alpha}(\eps))$ and thus $\nu(\alpha, j)\in C^{j}(\P^{n},\beta_{r}(B_{\alpha}(\eps)).$ Hence
\begin{equation}\label{eq1}n-n_{\alpha}+1\leq j\leq n \Rightarrow \psi^{0,j}(\alpha)\in C^{j}(V_{\alpha}\times \P^{n},B_{\alpha}(\eps))\end{equation}
Moreover $N(\alpha, n-n_{\alpha}+1)$ is a $(n_{\alpha}-1)$-dimensional projective space contained in $\P^{n}\backslash \beta_{r}(B_{\alpha_{0}\ldots\alpha_{i}\alpha}(\eps))$ for every $(\alpha_{0},\ldots,\alpha_{i});$ thus by Lemma \ref{lemmadef} if $n-n_{\alpha}+1\leq j\leq n$ then the cohomology class of $\nu(\alpha, j)$ generates $H^{j}(\P^{n}, \beta_{r}(B_{\alpha}(\eps))).$ Hence it follows that for every $\overline{\alpha}=(\alpha_{0}\cdots\alpha_{i}\alpha)$ such that $V_{\overline{\alpha}}\neq \emptyset$
\begin{equation}\label{eq2} n-n_{\alpha}+1\leq j\leq n\Rightarrow [\psi^{0,j}(\alpha)_{|\overline{\alpha}}]\, \textrm{ generates }\, H^{j}(V_{\overline{\alpha}}\times \P^{n},B_{\overline{\alpha}}(\eps))=\Z_{2}\end{equation}
For every $\alpha_{0},\alpha_{1}\in A$ such that $V_{\alpha_{0}\alpha_{1}}\neq \emptyset$ we consider a curve $c_{\alpha_{0}\alpha_{1}}:I\to V_{\alpha_{0}}\cup V_{\alpha_{1}}$ such that $c_{\alpha_{0}\alpha_{1}}(i)=\eta_{\alpha_{i}},\, i=0,1;$ since $\Omega\backslash f_{\eps}^{-1}(\mathcal{D}^{g})$ has codimension two in $\Omega,$ then we may choose $c_{\alpha_{0}\alpha_{1}}$ such that for every $t\in I$ we have $f_{\eps}(c_{\alpha_{0}\alpha_{1}}(t))\in\mathcal{D}^{g}.$ Consider the $\R^{n-j+1}$-bundle $L_{j}^{g}(\alpha_{0}\alpha_{1})=c_{\alpha_{0}\alpha_{1}}^{*}f_{\eps}^{*}(\mathcal{L}_{j}^{-})^{g}$ over $I$ and its projectivization $P(L_{j}^{g}(\alpha_{0}\alpha_{1})).$ Then the natural map
$$P(L_{j}^{g}(\alpha_{0}\alpha_{1}))\to\P^{n}$$ defines a $(n-j+1)$-chain $T(\alpha_{0}\alpha_{1},j-1)$ in $\P^{n}.$ Let $\tau(\alpha_{0}\alpha_{1},j-1)\in C^{j-1}(\P^{n})$ be its Poincaré dual and define $\theta^{1,j-1}\in K_{0}^{1,j-1}$ by setting for every $\alpha_{0},\alpha_{1}$ with $V_{\alpha_{0}\alpha_{1}}\neq \emptyset$
$$\theta^{1,j-1}(\alpha_{0}\alpha_{1})=\beta_{r}^{*}\tau(\alpha_{0}\alpha_{1}, j-1).$$
Notice that $\partial T(\alpha_{0}\alpha_{1}, j-1)=N(\alpha_{0},j)+N(\alpha_{1},j),$ hence $d\tau(\alpha_{0}\alpha_{1},j-1)=\nu(\alpha_{0},j)+\nu(\alpha_{1}, j);$ it follows that
\begin{equation}\label{eq3} \delta_{K}\psi^{0,j}=d_{K}\theta^{1,j-1}.\end{equation}
Moreover by construction if $n-n_{\alpha_{0}}+1\leq j\leq n$ and $n-n_{\alpha_{1}}+1\leq j\leq n,$ which by the previous computations implies $n-n_{\alpha_{0}\alpha_{1}}+1\leq j\leq n,$ then
\begin{equation}\label{eq4} \theta^{1,j-1}(\alpha_{0}\alpha_{1})\in C^{j-1}(V_{\alpha_{0}\alpha_{1}}\times \P^{n},B_{\alpha_{0}\alpha_{1}}(\eps)).\end{equation}
We compute now $\delta_{K}\theta^{1,j-1};$ first we recall the following fact, which is a direct consequence of the definition of Stiefel-Whitney classes.

\begin{fact}Let $\pi:E\to S^{1}$ a $\R^{k+1}$ fiber bundle and $P(\pi):P(E)\to S^{1}$ its projectivization. Moreover let $L:P(E)\to \P^{m},\, m>k$ be a linear map, $c\in H^{k}(P(E))$ such that for every $y \in S^{1}$ the class $c_{|P(E_{y})}$ generates $H^{k}(P(E_{y}))$ and $b\in H^{k+1}(\P^{m})$ be the generator. Then, writing $w_{1}(E)$ for the first Stiefel-Whitney class of $E$, the following holds: $$L^{*}b=P(\pi)^{*}w_{1}(E)\smile c.$$
\end{fact}

Let now $(\alpha_{0}\alpha_{1}\alpha_{2})=\bar\alpha$ such that $V_{\bar\alpha}\neq \emptyset.$ Then the curves $c_{\alpha_{0}\alpha_{1}}, c_{\alpha_{1}\alpha_{2}}$ and $c_{\alpha_{2}\alpha_{0}}$ define a map $\sigma_{\alpha_{0}\alpha_{1}\alpha_{2}}:S^{1}\to \Omega$ and we have the bundle $L_{j}^{g}(\alpha_{0}\alpha_{1}\alpha_{2})=\sigma_{\alpha_{0}\alpha_{1}\alpha_{2}}^{*}f_{\eps}^{*}(\mathcal{L}_{j}^{-})^{g}$ and its projectivization $P(L_{j}^{g}(\alpha_{0}\alpha_{1}\alpha_{2}))$ over $S^{1}.$ The natural map
$$P(L_{j}^{g}(\alpha_{0}\alpha_{1}\alpha_{2}))\to \P^{n}$$ defines a $(n-j+1)$-cochain which by construction equals $\delta_{K}\theta^{1,j-1}(\alpha_{0}\alpha_{1}\alpha_{2}).$ Thus by Fact 1 we have:
\begin{equation}\label{eq5}\delta_{K}\theta^{1,j-1}(\alpha_{0}\alpha_{1}\alpha_{2})=w_{1}(\partial(\alpha_{0}\alpha_{1}\alpha_{2}))(\psi^{0,j-1}(\alpha_{2})_{|\alpha_{0}\alpha_{1}\alpha_{2}})+dr^{2,j-1}(\alpha_{0}\alpha_{1}\alpha_{2}) \end{equation}
where $w_{1}(\partial(\alpha_{0}\alpha_{1}\alpha_{2}))=w_{1}(L_{j}^{g}(\alpha_{0}\alpha_{1}\alpha_{2})).$
Let now $\xi^{i}\in F_{1}^{i,j}(\eps, \mathfrak{U});$ we define $\xi^{i,0}\in K_{0}^{i,0}$ by $$\xi^{i,0}(\alpha_{0}\ldots\alpha_{i})\equiv \xi^{i}(\alpha_{0}\ldots \alpha_{i})$$ i.e. the values of $\xi^{i,0}(\alpha_{0}\ldots \alpha_{i})$ on every $0$-chain equals $\xi^{i}(\alpha_{0}\ldots \alpha_{i})\in \Z_{2}.$ Notice that by construction $d_{K}\xi^{i,0}=0$ and that
\begin{equation}\label{eq6} d_{1}\xi^{i}=0\,\,\Rightarrow\,\, \delta_{K}\xi^{i,0}=0.\end{equation}
Pick now $x\in F_{2}^{i,j}(\eps, \mathfrak{U})$ and $\xi^{i}$ such that $x=[\xi^{i}]_{\check{\delta}};$ consider the cochain $\xi^{i,0}\cdot \psi^{0,j}\in K_{0}^{i,j}.$ Since $\xi^{i}\in F_{1}^{i,j}(\eps, \mathfrak{U}),$ then (\ref{eq1}) implies
$$\xi^{i,0}\cdot \psi^{0,j}\in F_{0}^{i,j}(\eps, \mathfrak{U}).$$
Moreover by (\ref{eq2}) it follows that $[\xi^{i,0}\cdots \psi^{0,j}]_{1}=\xi^{i}$ and thus
$$[\xi^{i,0}\cdot \psi^{0,j}]_{2}=x.$$
We calculate now: \begin{align*}\delta_{F}(\xi^{i,0}\cdot \psi^{0,j})&=\delta_{K}(\xi^{i,0}\cdot \psi^{0,j})=\xi^{i,0}\cdot \delta_{K}\psi^{0,j}=\xi^{i,0}\cdot d_{K}\theta^{1,j-1}\\ &=d_{K}(\xi^{i,0}\cdot \theta^{1,j-1})=d_{F}(\xi^{i,0}\cdot \theta^{1,j-1}).\end{align*}
The first equality comes from $F_{0}^{i,j}(\eps, \mathfrak{U})\subset K_{0}^{i,j};$ the second from $\check{\delta}\xi^{i}=0;$ the third from (\ref{eq3}); the fourth from (\ref{eq6}); the last by $\xi^{i,0}\cdot\theta^{1,j-1}\in F_{0}^{i+1,j-1}(\eps, \mathfrak{U}),$ which is a direct consequence of (\ref{eq4}).\\
We finally compute $d_{2}(\eps, \mathfrak{U})(x)=[\delta_{F}(\xi^{i,0}\cdot \theta^{1,j-1})]_{2}:$
$$\delta_{F}(\xi^{i,0}\cdot \theta^{1,j-1})=\delta_{K}(\xi^{i,0}\cdot \theta^{1,j-1})=\xi^{i,0}\cdot \delta_{K}\theta^{1, j-1}$$
and thus by (\ref{eq5}) we have
$$[\delta_{F}(\xi^{i,0}\cdot \theta^{1,j-1})]_{1}(\alpha_{0}\cdots\alpha_{i+2})=\xi^{i}(\alpha_{0}\cdots\alpha_{i})w_{1}(\partial (\alpha_{i}\alpha_{i+1}\alpha_{i+2})).$$
Now we define $(F_{r}(\eps), d_{r}(\eps))$ to be the pure Leray spectral sequence of the map $(\beta_{l})_{|B(\eps)},$ which by definition is
$$(F_{r}(\eps),d_{r}(\eps))=\varinjlim_{\mathfrak{W}}\{(F_{r}(\eps, \mathfrak{W}), d_{r}(\eps, \mathfrak{W}))\}$$ where the direct limit is taken over all open covers of $\Omega$ directed by refinement. By Lemma \ref{lemmacover} the previous direct limit can be computed over $f_{\eps}$-regular covers; moreover by Corollary \ref{ac} for a $f_{\eps}$- regular cover $\mathfrak{U}$ the limit map gives natural isomorphism $$F_{2}^{i,j}(\eps, \mathfrak{U})\simeq \varinjlim_{\mathfrak{U}}\{F_{2}^{i,j}(\eps, \mathfrak{U})\}$$ Now Lemma \ref{lemmadef} implies the third of the following equalities:
$$F_{2}^{i,j}(\eps)=\varinjlim_{\mathfrak{W}}\{F_{2}^{i,j}(\eps, \mathfrak{W})\}=\varinjlim_{\mathfrak{U}}\{F_{2}^{i,j}(\eps, \mathfrak{U})\}=\check{H}^{i}(\Omega, \Omega_{n-j}(\eps);\Z_{2})$$
(we stress that the previous limits are attained at $f_{\eps}$-regular covers).\\
Then by definition of Cech cohomology class and of the connecting homomorphism we get that the differential $d_{2}(\eps):F_{2}^{i,j}(\eps)\to F^{i+2,j-1}(\eps)$ is given by
$$d_{2}(\eps)(x)=(x\smile f_{\eps}^{*}\partial^{*}w_{1}(\L_{j}^+)^{g})|_{(\Omega, \Omega_{n-j+1}(\eps))}$$ (here we are using the fact that $w_{1}(\L_{j}^{+})=w_{1}(\L_{j}^{-})$).\\
Consider now, for $\eps>0,$ the complex $(F_{0}(\eps), D_{\eps}=d+\delta)$.  Then for $\eps_{1}<\eps_{2}$ the inclusion $B(\eps_{2})\hookrightarrow B(\eps_{1})$ defines a morphism of filtered differential graded modules $i_{0}(\eps_{1},\eps_{2}):(F_{0}(\eps_{1}), D(\eps_{1}))\to (F_{0}(\eps_{2}), D(\eps_{2}))$ turning $\{(F_{0}(\eps), D(\eps))\}_{\eps>0}$ into an inverse system and thus $\{(F_{r}(\eps), d_{r}(\eps))\}_{\eps>0}$ into an inverse system of spectral sequences. We set
$$(F_{r},d_{r})\doteq \varprojlim_{\eps}\{(F_{r}(\eps), d_{r}(\eps))\}.$$
We examine $i_{2}(\eps_{1},\eps_{2}):F_{2}^{i,j}(\eps_{1})\to F_{2}^{i,j}(\eps_{2});$ it is readily verified that for $i,j\geq0$ the map $i_{2}(\eps_{1},\eps_{2})_{i,j}:F_{2}^{i,j}(\eps_{1})\to F_{2}^{i,j}(\eps_{2})$ equals the map
$$i^{*}(\eps_{1},\eps_{2}):\check{H}^{i}(\Omega, \Omega_{n-j}(\eps_{1}))\to \check{H}^{i}(\Omega, \Omega_{n-j}(\eps_{2}))$$ given by the inclusion of pairs $(\Omega, \Omega_{n-j}(\eps_{2}))\hookrightarrow (\Omega, \Omega_{n-j}(\eps_{1})).$ By semialgebraicity $i^{*}(\eps_{1},\eps_{2})$ is definitely an isomorphism, hence $i_{2}(\eps_{1},\eps_{2})$ is definitely an isomorphism and thus $i_{\infty}(\eps_{1}, \eps_{2})$ and $i_{0}^{*}(\eps_{1},\eps_{2}):H_{D}^{*}(F_{0}(\eps_{1}))\to H_{D}^{*}(F_{0}(\eps_{2}))$ are definitely isomorphisms. Thus $(F_{r},d_{r})$ converges to $\lim_{\eps}H^{*}(\Omega\times \P^{n}, B(\eps)).$ Again by semialgebraicity the inclusions $(\Omega \times \P^{n},B(\eps_{2}))\hookrightarrow (\Omega \times \P^{n}, B(\eps_{1}))$ are definitely homotopy equivalences and since for small $\eps>0$ the inclusion $B(\eps)\hookrightarrow B$ is a homotopy equivalence too, then we have that $(F_{r},d_{r})$ converges to
$$\varprojlim_{\eps}\{H^{*}(\Omega \times \P^{n}, B(\eps))\}=H^{*}(\Omega\times \P^{n},B).$$
It remains to identify $(F_{2},d_{2}).$ We already proved, in Theorem \ref{basiccoo} that $$\varinjlim_{\eps}\{H_{*}(\Omega, \Omega_{n-j}(\eps))\}=H_{*}(\Omega, \Omega^{j+1}).$$ Then, as before, the following chain of isomorphisms $\varprojlim\{H^{i}(\Omega, \Omega_{n-j}(\eps);\Z_{2})\}\simeq(\varinjlim\{H_{i}(\Omega, \Omega_{n-j}(\eps);\Z_{2})\})^{*}=(H_{i}(\Omega, \Omega^{j+1};\Z_{2}))^{*}$ gives $$F_{2}^{i,j}=H^{i}(\Omega, \Omega^{j+1};\Z_{2}).$$
It is easy to see that this spectral sequence and that one constructed in Theorem \ref{basiccoo} are isomorphic.\\
We define $\Gamma_{1,j}\in H^{2}(\Q,\D_{j})$ by
$$\Gamma_{1,j}^{g}\doteq \partial^{*}w_{1}(\L_{j}^{+})^{g}.$$
Consider now the following sequences of maps:
$$H^{2}(\mathcal{Q}, \mathcal{D}^{g}_{j})\stackrel{f_{\eps}^{*}}{\longrightarrow}H^{2}(\Omega, D^{g}_{j}(\eps))\stackrel{r_{\eps}^{*}}{\longrightarrow}H^{2}(\Omega,\Omega_{n-j+1}(\eps)\backslash \Omega_{n-j}(\eps)).$$
Notice that $r_{\eps}^{*}f_{\eps}^{*}\Gamma_{1,j}^{g}$ does not depend on $g$ and thus the differential $d_{2}(\eps)$ is given by $$x\mapsto (x\smile f_{\eps}^{*}\Gamma_{1,j}^{g})|_{(\Omega, \Omega_{n-j+1}(\eps))}$$ for \emph{any} $g.$
Let now $g=q_{0};$ then in this case $D_{j}^{q_{0}}=D_{j}^{q_{0}}(\eps)$ and $f^{*}=f^{*}_{\eps}.$ Consider the following commutative diagram of inclusions:
$$\begin{tikzpicture}[xscale=4.5, yscale=2]

    \node (A0_0) at (0, 0) {$(\Omega, \Omega_{n-j+1}(\eps))$};
    \node (A1_0) at (1, 0) {$(\Omega, \Omega_{n-j}(\eps)\cup D_{j}^{q_{0}})$};
    \node (A0_1) at (0, 1) {$(\Omega, \Omega^{j})$};
    \node (A1_1) at (1, 1) {$(\Omega, \Omega^{j+1}\cup D_{j}^{q_{0}})$};
    \path (A0_0) edge [->] node [auto] {$j(\eps)$} (A1_0);
    \path (A0_0) edge [->] node [auto,swap ] {$\rho(\eps)$} (A0_1);
    \path (A0_1) edge [->] node [auto] {$j$} (A1_1);
    \path (A1_0) edge [->] node [auto] {$\hat{\rho}(\eps)$} (A1_1);
      \end{tikzpicture}
$$

Then, using $\rho(\eps)$ also for the inclusion $(\Omega, \Omega_{n-j}(\eps))\hookrightarrow (\Omega, \Omega^{j+1})$ and setting $\gamma_{1,j}=f^{*}\Gamma_{1,j}^{q_{0}},$ we have for $x\in H^{i}(\Omega, \Omega^{j+1})$ the following chain of equalities:

\begin{align*}\rho(\eps)^{*}((x\smile \gamma_{1,j})|_{(\Omega, \Omega^{j})})&=\rho(\eps)^{*}j^{*}(x\smile f^{*}\Gamma_{1,j}^{q_{0}})=j(\eps)^{*}\hat{\rho}(\eps)^{*}(x\smile f^{*}\Gamma_{1,j}^{q_{0}})\\
&=j(\eps)^{*}(\rho(\eps)^{*}x\smile f_{\eps}^{*}\Gamma_{1,j}^{q_{0}})=d_{2}(\eps)(\rho(\eps)^{*}x).
\end{align*}

This proves that the following diagram is commutative:

$$\begin{tikzpicture}[xscale=4.5, yscale=2]

    \node (A0_0) at (0, 0) {$H^{i}(\Omega, \Omega_{n-j}(\eps))$};
    \node (A1_0) at (1, 0) {$H^{i+2}(\Omega, \Omega_{n-j+1}(\eps))$};
    \node (A0_1) at (0, 1) {$H^{i}(\Omega, \Omega^{j+1})$};
    \node (A1_1) at (1, 1) {$H^{i+2}(\Omega, \Omega^{j})$};
    \path (A0_0) edge [->] node [auto] {$d_{2}(\eps)$} (A1_0);
    \path (A0_1) edge [->] node [auto,swap ] {$\rho(\eps)^{*}$} (A0_0);
    \path (A0_1) edge [->] node [auto] {$(\cdot \smile \gamma_{j})|_{(\Omega,\Omega^{j})}$} (A1_1);
    \path (A1_1) edge [->] node [auto] {$\rho(\eps)^{*}$} (A1_0);
      \end{tikzpicture}
$$

From this the conlusion follows.
\end{proof}

We are now ready to prove the statement concerning the second differential of the spectral sequence of Theorem A.\\
First we fix a scalar product (i.e. a positive definite form) $q_{0}.$\\
We let $\partial^{*}:H^{1}(D_{j})\to H^{2}(C\Omega, D_{j})$ be the connecting homomorphism  and we define $\gamma_{1,j}\in H^{2}(C\Omega, D^{j})$ by
$$\gamma_{1,j}=\partial^{*}\bar{p}^{*}w_{1}(\L_{j}^{+})$$
(notice that all the previous objects are those associated to $q_{0}$ and that $\gamma_{1,j}=\bar{p}^{*}\phi_{j}$ as defined in the Introduction).
\begin{thmb} \label{final} For every $i,j\geq 0$ the differential $d_{2}:E_{2}^{i,j}\to E_{2}^{i+2,j-1}$is given by: $$d_{2}(x)=(x\smile \gamma_{1,j})|_{(C\Omega, \Omega^{j})}.$$

\begin{proof}
As before we replace now $K$ with $\hat{K}=(-\infty, 0]\times K,$ the map $p$ with the map $\hat{p}=(q_{0},p):\R^{n+1}\to \R^{k+2}$ and we apply the previoius Theorem to $(\hat{p}, \hat{K})$. As before we use the deformation retraction $(\hat{\Omega},\hat{\Omega}^{j+1})\to (\hat{\Omega}, \Omega^{j+1})=(C\Omega, \Omega^{j+1}).$  Notice that we have also the deformation retraction $$r:(\hat{\Omega}, \hat{D}_{j})\to(\hat{\Omega},D_{j})$$ where $D_{j}$ is identified with $\hat{D}_{j}\cap\{\eta=0\}:$  by definition $\omega\in D_{j}$ if and only if $(\eta,\omega)\in \hat{D}_{j}$ and for every $0<j<n+1$ we have $(1,0,\ldots,0)\notin D_{j}$ since all the eigenvalues of $\langle(1,0,\ldots,0), \hat{p}\rangle =-q_{0}$ with respect to $q_{0}$ coincide. Then by construction
$$r^{*}\gamma_{1,j}= \partial^{*}\bar{p}^{*}w_{1}(\L_{j}^{+})$$
and by naturality the conclusion follows.

\end{proof}

\end{thmb}

\section{Some remarks on spectral sequences}

Here we make some remarks which will be useful in the sequel. We always make use of $\Z_{2}$ coefficients, in order to avoid  sign problems; the following results still hold for $\Z$ coefficients, but sign must be put appropriately. We begin with the following.

\begin{lemma}\label{contraction}Let $(C_{*},\partial_{*})$ be an acyclic free chain complex and $(D_{*},\partial_{*}^{D})$ be an acyclic subcomplex. Then there exists a chain homotopy $$K_{*}:C_{*}\to C_{*+1}$$ such that $\partial_{*+1} K_{*}+K_{*-1}\partial_{*} =I_{*}$ and $K_{*}(D_{*})\subset(D_{*+1}).$
\end{lemma}

\begin{proof} By taking a right inverse $s_{q-1}^{D}$ of $\partial_{q}^{D}$, %$$s_{q-1}^{D}:Z_{q-1}^{D}=B_{q-1}^{D}\to D_{q}$$
which exists since $D_{q-1}$ and hence $Z_{q-1}^{D}$ are free, a chain contraction $K_{q}^{D}$ for $D$ is defined by: $K_{q}^{D}=s_{q}^{D}(I_{q}-s_{q-1}^{D}\partial_{q}^{D})$. Since $Z_{q}$ is free, then it is possible to extend $s_{q-1}^{D}$ to a right inverse $s_{q-1}$ of $\partial_{q}$: $$s_{q-1}:Z_{q-1}=B_{q-1}\to C_{q}.$$ Then by setting $$K_{q}=s_{q}(I_{q}-s_{q-1}\partial_{q})$$ we obtain a chain contraction for the complex $(C_{*},\partial_{*})$ which restricts to a chain contraction for the subcomplex $(D_{*},\partial_{*}^{D}).$
\end{proof}

Let now $X$ be a topological space and $Y$ be a subspace. If we consider an open cover $\mathfrak{U}=\{V_{\alpha}\}_{\alpha \in A}$ for $X,$ then the Mayer-Vietoris bicomplex $E(Y)^{*,*}_{0}$ for the pair $(X,Y)$ relative to the cover $\mathfrak{U}$ is defined by $$E(Y)_{0}^{p,q}=\check{C}^{p}(\mathfrak{U},\mathfrak{U}\cap Y; C^{q})=\prod_{\alpha_{0}<\cdots<\alpha_{p}}C^{q}(V_{\alpha_{0}\cdots \alpha_{p}}, V_{\alpha_{0}\cdots \alpha_{p}}\cap Y)$$ with the horizontal $\delta$ and the vertical $d$ defined as the usual ones. By the Mayer-Vietoris principle, which of course extends to the case of a subspace pair, each row of the augmented chain complex of $E(Y)_{0}^{*,*}$ is exact, i.e. for each $q\geq0$ the chain complex
$$0\to C_{\mathfrak{U}}^{q}(X,Y)\to\check{C}^{0}(\mathfrak{U}, \mathfrak{U}\cap Y, C^{q})\to \cdots$$ is acyclic - we recall that $(C_{\mathfrak{U}}^{*}(X,Y),d)$ is defined to be the complex of $\mathfrak{U}$-small singular cochains and that the following isomorphism holds: $$H_{d}(C^{*}_{\mathfrak{U}}(X,Y))\simeq H^{*}(X,Y).$$
From this it follows that the spectral sequence associated to $E(Y)_{0}^{*,*}$ converges to $$H^{*}(X,Y)\simeq H^{*}_{D}(E(Y)_{0}^{*,*})$$ where $H^{*}_{D}(E(Y)_{0}^{*,*})$ is the cohomology of the complex $E(Y)_{0}^{*,*}$ with differential $D=d+\delta.$ We also recall that $$r^{*}:C^{*}_{\mathfrak{U}}(X,Y)\to \check{C}^{*}(\mathfrak{U},\mathfrak{U}\cap Y, C^{*})$$ induces isomorphisms on cohomologies; if we take a chain contraction $K$ for the Mayer-Vietoris rows of the pair $(X,Y),$ then we can define a homotopy inverse $f$ to $r^{*}$ by the following procedure. If $c=\sum_{i=0}^{n}c_{i}$ and $Dc=\sum_{i=0}^{n+1}b_{i}$ then we set $$f(c)=\sum_{i=0}^{n}(dK)^{i}c_{i}+\sum_{i=0}^{n+1}K(dK)^{i-1}b_{i}.$$
The reader can see \cite{Bott} for more details.\\
If we let $Z\subset Y$ be a subspace, then $E(Y)_{0}^{*,*}$ is naturally included in the Mayer-Vietoris bicomplex $E(Z)_{0}^{*,*}$ for the pair $(X,Z)$ relative to the cover $\mathfrak{U}:$ $$i_{0}:E(Y)_{0}^{*,*}\hookrightarrow E(Z)_{0}^{*,*}.$$
Since $i_{0}$ obviously commutes with the total differentials, then it induces a morphism of spectral sequence, and thus a map $$i_{0}^{*}:H^{*}_{D}(E(Y)_{0})\to H^{*}_{D}(E(Z)_{0}).$$ At the same time the inclusion $j:(X,Z)\hookrightarrow (X,Y)$ induces a map $$j^{*}:H^{*}(X,Y)\to H^{*}(X,Z).$$ With the previous notations we prove the following useful lemma.

\begin{lemma}\label{ssc}
There are isomorphisms $f_{Y}^{*}:H^{*}_{D}(E(Y)_{0})\to H^{*}(X,Y)$ and $f_{Z}^{*}:H^{*}_{D}(E(Z)_{0})\to H^{*}(X,Z)$ such that the following diagram is commutative:
$$\begin{tikzpicture}[xscale=3, yscale=2]

    \node (A0_0) at (0, 0) {$H^{*}(X,Y)$};
    \node (A1_0) at (1, 0) {$ H^{*}(X,Z)$};
    \node (A0_1) at (0, 1) {$H^{*}_{D}(E(Y)_{0})$};
    \node (A1_1) at (1, 1) {$H^{*}_{D}(E(Z)_{0})$};
    \path (A0_0) edge [->] node [auto] {$j^{*}$} (A1_0);
    \path (A0_1) edge [->] node [auto,swap ] {$f_{Y}^{*}$} (A0_0);
    \path (A0_1) edge [->] node [auto] {$i_{0}^{*}$} (A1_1);
    \path (A1_1) edge [->] node [auto] {$f_{Z}^{*}$} (A1_0);
      \end{tikzpicture}
$$
\end{lemma}
\begin{proof}The augmented Mayer-Vietoris complex for the pair $(X,Y)$ relative to $\mathfrak{U}$ is a subcomplex of the augmented Mayer-Vietoris complex for the pair $(X,Z)$ relative to $\mathfrak{U}.$ Thus by Lemma \ref{contraction} for every $q\geq 0$ there exists a chain contraction $K_{Z}$ for the complex $$0\to C^{q}_{\mathfrak{U}}(X,Z)\to\check{C}^{0}(\mathfrak{U}, \mathfrak{U}\cap Z, C^{q})\to \cdots$$ which restricts to a chain contraction $K_{Y}$ for the complex $$0\to C^{q}_{\mathfrak{U}}(X,Y)\to\check{C}^{0}(\mathfrak{U}, \mathfrak{U}\cap Y, C^{q})\to \cdots$$
We define $f_{Y}$ and $f_{Z}$ with the above construction and we take $f_{Y}^{*}$ and $f_{Z}^{*}$ to be the induced maps in cohomology. Then  $f_{Z}$ restricted to $E(Y)_{0}^{*,*}$ coincides with $f_{Y}$ and since $j^{*}$ is induced by the inclusion $j^{\natural}:C_{\mathfrak{U}}^{q}(X,Y)\to C_{\mathfrak{U}}^{q}(X,Z),$ then the conclusion follows.
\end{proof}

\begin{remark}
Notice that $i_{0}:E(Y)_{0}^{*,*}\to E(Z)_{0}^{*,*}$ induces maps of spectral sequences respecting the bigradings $(i_{r})_{a,b}:E(Y)_{r}^{a,b}\to E(Z)_{r}^{a,b}$ and thus also a map $i_{\infty}:E(Y)_{\infty}\to E(Z)_{\infty}.$ Even tough $E(Y)_{\infty}\simeq H^{*}(X,Y)$ and $E(Z)_{\infty}\simeq H^{*}(X,Z),$ in general $i_{\infty}$ does not equal $j_{*}$ (neither their ranks do); the same considerations hold for the more general case of a map of pairs $f:(X,Y)\to(X',Y').$
\end{remark}
We recall also the following fact. Given a first quadrant bicomplex $E_{0}^{*,*}$ with total differential $D=d+\delta$ and associated convergent spectral sequence $(E_{r},d_{r})_{r\geq 0},$ then $$E_{\infty}^{*}\simeq G H^{*}_{D}(E_{0})$$ and there is a canonical homomorphism $$p_{E}:H_{D}^{*}(E_{0})\to E_{\infty}^{0,*}$$ constructed as follows.
Let $[\psi]_{D}\in H_{D}^{k}(E_{0});$ then there exists $\psi_{i}\in E_{0}^{i,k-i}$ for $i=0,\ldots,k$ such that $D(\psi_{0}+\cdots+\psi_{k})=0$ and $$[\psi]_{D}=[\psi_{0}+\cdots+\psi_{k}]_{D}.$$
By definition of the differentials $d_{r}, r\geq 0,$ the element $\psi_{0}$ survives to $E_{\infty}.$ We check that the correspondence $$p_{E}:[\psi]_{D}\mapsto [\psi_{0}]_{\infty}$$ is well defined: since $\psi_{0}\in E_{0}^{0,k}$ and $E_{0}^{i,j}=0$ for $i<0,$ then $[\psi_{0}]_{\infty}=[\psi_{0}']_{\infty}$ if and only if $\psi_{0}$ and $\psi_{0}'$ survive to $E_{\infty}$ and $[\psi_{0}]_{1}=[\psi_{0}']_{1};$ if $\psi=\psi'+D \phi,$ then $\psi_{0}=\psi_{0}'+d\phi_{0}$ and thus $[\psi_{0}]_{1}=[\psi_{0}']_{1}.$

\section{Projective inclusion}

In this section we study the image of the homology of $X$ under the inclusion map $$j:X\to \P^{n}.$$
Using the above notations, we define $\hat{B}=\{(\hat{\omega}, x)\in \hat{\Omega} \times \P^{n}\, : \, (\hat{\omega}\hat{p})(x)>0\}$ and we call $(E_{r},d_{r})$ the spectral sequence of Theorem A converging to $H^{*}(\hat{\Omega}\times \P^{n},\hat{B}).$ Moreover we let $K_{0}^{*,*}$ be the Leray bicomplex for the map $\hat{\Omega}\times \P^{n}\to \hat{\Omega}$ (it equals the Kunneth bicomplex for $\hat{\Omega}\times \P^{n}$). Thus there is a morphism of spectral sequence $(i_{r}:E_{r}\to K_{r})_{r\geq0}$ induced by the inclusion $j:(\hat{\Omega}\times \P^{n},\emptyset)\to (\hat{\Omega}\times \P^{n},B).$ With the above notations we prove the following theorem which gives the rank of the homomorphism $$j_{*}:H_{*}(X)\to H_{*}(\P^{n}).$$

\begin{thmc}\label{projincl}
For every $b\in \Z$ the following holds: $$\emph{\textrm{rk}} (j_{*})_{b}=\emph{\textrm{rk}} (i_{\infty})_{0,n-b}.$$
Moreover the map $(i_{\infty})_{0,n-b}:E_{\infty}^{0,n-b}\to K_{\infty}^{0,n-b}=\Z_{2}$ is an isomorphism onto its image.

\end{thmc}
\begin{proof}

First we look at the following commutative diagram of maps

$$\begin{tikzpicture}[xscale=4, yscale=2]

    \node (A0_0) at (0, 0) {$H_{b}(X)$};
    \node (A0_1) at (0, 1) {$ H^{n-b}(\P^{n},\P^{n}\backslash X)$};
    \node (A0_2) at (0, 2) {$H^{n-b}(\hat{\Omega}\times \P^{n},\hat{B})$};

    \node (A1_0) at (1, 0) {$H_{b}(\P^{n})$};
    \node (A1_1) at (1, 1) {$H^{n-b}(\P^{n})$};
    \node (A1_2) at (1, 2) {$H^{n-b}(\hat{\Omega}\times \P^{n})$};
        \path (A0_0) edge [->] node [auto] {$P^{*}$} (A0_1);
    \path (A0_1) edge [->] node [auto] {$\beta_{l}^{*}$} (A0_2);
    \path (A0_0) edge [->] node [auto] {$(j_{*})_{b}$} (A1_0);
    \path (A0_1) edge [->] node [auto] {$({j'}^{*})_{n-b}$} (A1_1);
    \path (A0_2) edge [->] node [auto] {$(j^{*})_{n-b}$} (A1_2);
       \path (A1_0) edge [->] node [auto, swap] {$P^{*}$} (A1_1);
    \path (A1_1) edge [->] node [auto, swap] {$\beta_{l}^{*}$} (A1_2);

      \end{tikzpicture}
$$

where the maps $j_{*},j^{*}$ and ${j'}^{*}$ are those induced by inclusions and the $P^{*}$'s are Poincaré duality isomorphisms; commutativity follows from naturality of Poincaré duality.
Since $\hat{\Omega}\approx C\Omega,$ then it is contractible and $\beta_{l}:(\hat{\Omega}\times \P^{n},\hat{B})\to (\P^{n},\P^{n}\backslash X)$ is a homotopy equivalence; hence all the vertical arrows are isomorphisms. Thus we identify $(j_{*})_{b}$ with $(j^{*})_{n-b}.$\\
Let now $\eps>0$ such that $\hat{C}(\eps)\hookrightarrow \hat{B}$ is a homotopy equivalence (we use the above notations); then the inclusion of pairs
$$(\hat{\Omega}\times \hat \P^{n}, \hat{C}(\eps))\stackrel{\hat{j}(\eps)}{\longrightarrow}(\hat{\Omega}\times \P^{n}, \hat{B})$$
also is a homotopy equivalence and the inclusion $(\hat{\Omega}\times \P^{n},\emptyset)\stackrel{j}{\longrightarrow}(\hat{\Omega}\times \P^{n}, \hat{B})$ factors trough:

$$\begin{tikzpicture}[xscale=2, yscale=2]

    \node (A2_0) at (2, 0) {$(\hat{\Omega}\times\P^{n},\hat{C}(\eps))$};
    \node (A1_1) at (1, 1) {$(\hat{\Omega}\times\P^{n},\emptyset)$};
    \node (A3_1) at (3, 1) {$(\hat{\Omega}\times\P^{n},\hat{B}(\eps))$};

            \path (A1_1) edge [->] node [auto] {$j$} (A3_1);
    \path (A1_1) edge [->] node [auto,swap] {$j(\eps)$} (A2_0);
    \path (A2_0) edge [->] node [auto, swap] {$\hat{j}(\eps)$} (A3_1);

      \end{tikzpicture}
$$
Since $\hat{j}(\eps)$ is a homotopy equivalence, it follows that:
$$\textrm{rk} (j^{*})_{n-b}=\textrm{rk} (j(\eps)^{*})_{n-b}.$$
Let now $\mathfrak{U}$ be any cover of $\hat{\Omega}$ and consider the Leray-Mayer-Vietoris bicomplexes $\hat{F}^{*,*}(\eps, \mathfrak{U})$ and $K_{0}^{*,*}(\mathfrak{U})$ with their respective associated spectral sequences; since $i_{0}(\eps, \mathfrak{U}):\hat{F}_{0}^{*,*}(\eps, \mathfrak{U})\hookrightarrow K_{0}^{*,*}(\mathfrak{U})$ there is a morphism of respective spectral sequences. Moreover by Mayer-Vietoris argument, the spectral sequence $(\hat{F}_{r}(\eps, \mathfrak{U}), \hat{d}_{r}(\eps, \mathfrak{U}))_{r\geq0}$ converges to $H^{*}(\hat{\Omega}\times \P^{n},\hat{C}(\eps))$ and $(K_{r}(\mathfrak{U}), d_{r}(\mathfrak{U}))_{r\geq0}$ converges to $H^{*}(\hat{\Omega}\times \P^{n}, \emptyset).$
We look now at the following commutative diagram:

$$\begin{tikzpicture}[xscale=4.5, yscale=2]

    \node (A0_0) at (0, 0) {$H^{n-b}(\hat{\Omega}\times\P^{n},\hat{B})$};
    \node (A0_1) at (0, 1) {$ H_{D}^{n-b}(E_{0}(\eps, \mathfrak{U}))$};
    \node (A0_2) at (0, 2) {$E_{\infty}^{0,n-b}(\eps, \mathfrak{U})$};

    \node (A1_0) at (1, 0) {$H^{n-b}(\hat{\Omega}\times\P^{n})$};
    \node (A1_1) at (1, 1) {$H_{D}^{n-b}(K_{0}(\mathfrak{U}))$};
    \node (A1_2) at (1, 2) {$K_{\infty}^{0,n-b}(\mathfrak{U})$};
        \path (A0_0) edge [->] node [auto] {$(f^{*}_{E})^{-1}$} (A0_1);
    \path (A0_1) edge [->] node [auto] {$p_{E}(\eps, \mathfrak{U})$} (A0_2);
    \path (A0_0) edge [->] node [auto] {$(j^{*})_{n-b}$} (A1_0);
    \path (A0_1) edge [->] node [auto] {$(i_{0}^{*}(\eps, \mathfrak{U}))_{n-b}$} (A1_1);
    \path (A0_2) edge [->] node [auto] {$(i_{\infty}(\eps, \mathfrak{U}))_{0,n-b}$} (A1_2);
       \path (A1_0) edge [->] node [auto, swap] {$(f^{*}_{K})^{-1}$} (A1_1);
    \path (A1_1) edge [->] node [auto, swap] {$p_{K}$} (A1_2);

      \end{tikzpicture}
$$
The upper square is commutative, since if we let $\psi=\psi_{0}+\cdots+\psi_{n-b}\in E_{0}^{n-b}$ with $D\psi=0,$ then (avoiding the $(\eps, \mathfrak{U})$-notations, but only for the next formula): $$p_{K}(i_{0}^{*})_{n-b}[\psi]_{E}=p_{K}[\psi]_{K}=[\psi_{0}]_{\infty, K}=(i_{\infty})_{0,n-b}[\psi_{0}]_{\infty, E}=(i_{\infty})_{0,n-b}p_{E}[\psi]_{E}.$$
The lower square is the one coming from Lemma \ref{ssc} with the vertical arrows inverted, hence it is commutative.\\
Since $K_{\infty}(\mathfrak{U})=K_{2}(\mathfrak{U})$ has only one column (the first), then $p_{K}(\mathfrak{U}):H_{D}^{n-b}(K_{0}(\mathfrak{U}))\to K_{\infty}^{0,n-b}(\mathfrak{U})$ is an isomorphism, hence for $0\leq b\leq n$ and using the above identifications we can identify the map $(j_{*})_{b}:H_{b}(X)\to H_{b}(\P^{n})$ with $$(i_{\infty}(\eps, \mathfrak{U}))_{0,n-b} (p_{E}(\eps, \mathfrak{U}))_{n-b}:H_{D}^{n-b}(E_{0}(\eps, \mathfrak{U}))\to \Z_{2}.$$
Since $(p_{E}(\eps, \mathfrak{U}))_{n-b}$ is surjective, then:
$$\textrm{rk} (j^{*})_{n-b}=\textrm{rk} (i_{\infty}(\eps, \mathfrak{U}))_{0,n-b}.$$
By Corollary \ref{ac} and Lemma \ref{lemmacover} there exists a family $\mathcal{C}$ of covers which is cofinal in the family of all covers such that for every $\mathfrak{U}\in \mathcal{C}$ the natural map $\hat{F}^{i,j}_{2}(\eps, \mathfrak{U})\to \hat{F}_{2}^{i,j}(\eps)$ is an isomorphism.
\begin{remark} Since here  we do not need the cover to be convex, the existence of the family $\mathcal{C}$ follows from easier consideration. Let $h:\hat{\Omega} \to |K|\subset \R^{N}$ be a triangulation respecting the filtration $\{\hat{\Omega}_{j}\}_{j=0}^{n+2},$ and $\mathfrak{W}$ be a cover of $\hat{\Omega}.$ Let $\mathfrak{U}'$ be a \emph{convex} cover of $|K|$ refining $h(\mathfrak{ W})$ and such that for every $U'\in \mathfrak{U}'$ the intersection $h(\hat{\Omega}_{j})\cap U'$ is contractible  for every $j$ (the existence of such a $\mathfrak{U}'$ follows from the fact that $h(\hat{\Omega}_{j})$ is a subcomplex of $|K|).$ Then the cover $\mathfrak{U}=h^{-1}(\mathfrak{U}')$  refines $\mathfrak{W}$ and since for every $j$ and $U\in \mathfrak{U}$ the intersection $\hat{\Omega}_{j}\cap U$ is contractible, then the natural map $\hat{F}^{i,j}_{2}(\eps, \mathfrak{U})\to \hat{F}_{2}^{i,j}(\eps)$ is an isomorphism. \end{remark}
It follows that $\textrm{rk} (i_{\infty}(\eps, \mathfrak{U})_{0,n-b})=\textrm{rk} (i_{\infty} (\eps))_{0,n-b},$ and thus by semialgebraicity we have
$$\textrm{rk}(i_{\infty}(\eps))_{0,n-b}=\textrm{rk} (i_{\infty})_{0,n-b}.$$
It remains to study the map $(i_{\infty})_{0,n-b}:E_{\infty}^{0,n-b}\to K_{\infty}^{0,n-b}=K_{2}^{0,n-b}.$ \\
If $E_{\infty}^{0,n-b}$ is zero, then $(i_{\infty})_{0,n-b}$ is obviously an isomorphism onto its image.\\
If $E_{\infty}^{0,n-b}$ is not zero then, since $E_{2}^{0,n-b}=H^{0}(C\Omega,\Omega^{n-b+1}),$ it must be $\hat{\Omega}^{n-b+1}=\emptyset$ and $$E_{\infty}^{0,n-b}=E_{2}^{0,n-b}=\Z_{2}.$$ From this it follows that $$i_{\infty}^{0,n-b}=i_{2}^{0,n-b}.$$
By looking directly at the two spectral sequences, we see that $i_{2}^{0,n-b}:H^{0}(\hat{\Omega},\emptyset)\to H^{0}(\hat{\Omega})\otimes H^{n-b}(\P^{n})$ is the identity and then the conclusion follows.
\end{proof}
We can immediately derive the following elementary corollary\begin{coro}
If $b>n-\mu$ then $(j_{*})_{b}=0.$
\end{coro}
\begin{proof}
Since $n-b<\mu$ then $\Omega^{n-b+1}\neq \emptyset.$ This gives $E_{2}^{0,n-b}=0$ and thus applying the previous theorem the conclusion follows.
\end{proof}

\section{Hyperplane section}

We consider here the following problem: given $X\subset \P^{n}$ defined by quadratic inequalities and $V$ a codimension one subspace of $\R^{n+1}$ with projectivization $\bar
V\subset \P^n$, determine the homology of $(X, X\cap \bar{V}).$\\
Thus let $p:\R^{n+1}\to \R^{k+1}\supseteq K$ be homogeneous quadratic and $X=p^{-1}(K)\subset \P^{n}.$ Let $h$ be a degree one homogeneous polynomial such that
$$V=\{h=0\}=\{h^{2}=0\}.$$
We can consider the function $\ii_{V}^{+}:\Omega \to \N$ defined by
$$\ii^{+}_{V}(\omega)=\ii^{+}(\omega p|_{V})$$
and we try describe the homology of $(X, X\cap \bar{V})$ only in terms of $\ii^{+}$ and $\ii^{+}_{V}.$\\
We introduce the quadratic map $p_{h}:\R^{n+1}\to \R^{k+2}$ defined by
$$p_{h}\doteq(p, h^{2}).$$
Then we have the following equalities:
$$X=p_{h}^{-1}(K\times \R)\quad \textrm{and}\quad X\cap \bar{V}=p_{h}^{-1}(K\times (-\infty,0]).$$
We consider $\hat{\Omega}=(K\times (-\infty,0])^{\circ}\cap S^{k+1},$ and the function $\ii_{h}^{+}:\R^{k+1}\times \R\to \N$ defined by
$$\ii^{+}_{h}(\omega, t)=\ii^{+}(\bar{p}_{h}(\omega,t))=\ii^{+}({\omega}p+th^{2}),\quad (\omega,t)\in\R^{k+1}\times \R.$$
For the moment we define, for $j\in \Z$ the set
$$\hat{\Omega}^{j+1}=\{\eta \in \hat{\Omega}\, : \, \ii_{h}^{+}(\eta)\geq j+1\}$$ and we identify $\Omega$ with $\{(\omega,t)\in \hat{\Omega}\, : \, t=0\}.$\\
With the previous notations we prove the following.

\begin{lemma}\label{relative}
There exists a cohomology spectral sequence $(G_{r},d_{r})$ of the first quadrant converging to $H_{n-*}(X, X\cap \bar V)$ such that
$$G_{2}^{i,j}=H^{i}(\hat{\Omega}^{j+1},\Omega^{j+1}).$$
\end{lemma}
\begin{proof}
Consider for $\eps>0$ the sets $C_{h}(\eps)=\{(\eta,x)\in \hat{\Omega}\times \P^{n}\, : \, (\eta p_{h})(x)\geq \eps\}$ and $C(\eps)=C_{h}(\eps)\cap \Omega\times \P^{n}.$
By semialgebraic triviality for small $\eps$ the inclusion $$(C_{h}(\eps), C(\eps))\hookrightarrow (B_{h},B)$$ is a homotopy equivalence (here $B_{h}$ stands for $\{(\eta,x)\in \hat{\Omega}\times \P^{n}\, : \, (\eta p_{h})(x)>0\}$ and $B$ for $B_{h}\cap \Omega \times \P^{n}$).\\
Consider the projection $\beta_{r}:\hat{\Omega}\times \P^{n}\to \P^{n};$ then $\beta_{r}(B_{h})=\P^{n}\backslash (X\cap H)$ and $\beta_{r}(B)=\P^{n}\backslash X;$ moreover by Lemma \ref{lemmahomeq} the previous are homotopy equivalences. Hence it follows:
$$H^{*}(C_{h}(\eps), C(\eps))\simeq H^{*}(B_{h},B)\simeq H^{*}(\P^{n}\backslash (X\cap H), \P^{n}\backslash X)\simeq H_{n-*}(X, X\cap H)$$
where the last isomorphism is given by Alexander-Pontryagin Duality.
Consider now $\beta_{l}:C_{h}(\eps)\to \hat{\Omega}.$ Then by Leray there is a cohomology spectral sequence $(G_{r}(\eps), d_{r}(\eps))$ converging to $H^{*}(C_{h}(\eps), C(\eps))$ such that
$$G_{2}^{i,j}=\check{H}^{i}(\hat{\Omega}, \mathcal{G}^{j}(\eps))$$
where $\mathcal{G}^{j}(\eps)$ is a sheaf such that for $\eta\in \hat{\Omega}$
$$(\mathcal{G}^{j}(\eps))_{\eta}=H^{j}(\beta_{l}^{-1}(\eta)\cap C_{h}(\eps), \beta_{l}^{-1}(\eta)\cap C(\eps))$$ (here we are using tha fact that both $C_{h}(\eps)$ and $C(\eps)$ are compact). We use now $\ii_{h}^{-}(\eps):\hat{\Omega}\to \N$ for the function $\eta\mapsto \ii^{-}(\eta p_{h} -\eps g)$ where $g$ is an arbitrary positive definite form, and we set $\hat{\Omega}_{n-j}(\eps)=\{\ii_{h}^{-}(\eps)\leq n-j\}.$
If $\eta\notin \Omega$, then $(\beta_{l}^{-1}(\eta)\cap C_{h}(\eps), \beta_{l}^{-1}(\eta)\cap C(\eps))\simeq(\P^{n-\ii_{h}^{-}(\eps)(\eta)},\emptyset);$ on the contrary if $\eta\in \Omega$ then $(\beta_{l}^{-1}(\eta)\cap C_{h}(\eps), \beta_{l}^{-1}(\eta)\cap C(\eps))=(\P^{n-\ii_{h}^{-}(\eps)(\eta)},\P^{n-\ii_{h}^{-}(\eps)(\eta)}).$ Since $\Omega$ is closed in $\hat{\Omega},$ it follows that
$$G_{2}^{j,j}(\eps)=\check{H}^{i}(\hat{\Omega}_{n-j}(\eps), \Omega_{n-j}(\eps)).$$
We define now $$(G_{r},d_{r})=\varprojlim_{\eps}\{(G_{r}(\eps),d_{r}(\eps))\}$$ and using the same argument as in the end of Theorem A we finally have
$$G_{2}^{i,j}=H^{i}(\hat{\Omega}^{j+1},\Omega^{j+1}).$$

\end{proof}

%Let $\ii_{V}^{+}:\Omega
%\to \N$ be the function $\omega \mapsto \ii^{+}((\omega p)_{|V});$
We are ready now for the proof of Theorem D; we define for $j>0$ the following set:
$$
\Omega_V^{j}=\{\omega \in \Omega : \mathrm i^{+}\left(\omega p|_V\right)\geq j\}.$$
%(by definition)
%$$\Omega_{V}^{0}\backslash A^{0}\doteq C\Omega.$$
%With the previous conventions we prove the following theorem.
\begin{thmd}
There exists a cohomology spectral sequence $(G_{r},d_{r})$ of the
first quadrant converging to $H_{n-*}(X_p, X_p\cap \bar V)$ such
that
$$
G_{2}^{i,j}=H^{i}(\Omega_V^{j},\Omega^{j+1}),\
j>0, \quad G_{2}^{i,0}=H^{i}(C\Omega,\Omega^1).$$
\end{thmd}

\begin{proof}
Take the spectral sequence $(G_{r},d_{r})$ to be that of lemma \ref{relative}; then it remains to prove that $G_{2}^{i,j}$ is isomorphic to the group described in the statement.\\
In the case $j=0$ we have that $\hat{\Omega}^{1}$ contains $(0,\ldots, 0,1)$ and, since $t_{1}\leq t_{2}$ implies $\ii_{h}(\omega,t_{1})\leq \ii_{h}^{+}(\omega, t_{2}),$ the set $\hat{\Omega}^{1}$ is contractible. Thus, using the long exact sequences of the pairs, we see that for every $i\geq 0$ the following holds:
$$G_{2}^{i,0}=H^{i}(\hat{\Omega}^{1},\Omega^{1})\simeq H^{i}(C\Omega, \Omega^{1}).$$
We study now the case $j>0.$\\
We identify $\hat{\Omega}\backslash \{(0,\ldots,0,1)\}$ with $\Omega \times [0,\infty)$ via the index preserving homeomorphism
$$(\omega,t)\mapsto (\omega,t)/\|\omega\|.$$
Thus, under the above identification, we have for $j>0$
$$\hat{\Omega}^{j+1}=\{(\omega,t)\in \Omega\times [0,\infty)\, : \, \ii_{h}^{+}(\omega,t)\geq j+1\}$$
and letting $\pi:\Omega \times [0,\infty)$ be the projection onto the first factor, we see that
$$\pi(\hat{\Omega}^{j+1})=\{\omega \, : \, \exists t>0 \,\,\,\textrm{s.t. } \ii_{h}^{+}(\omega,t)\geq j+1\}.$$
We prove that $\pi:\hat{\Omega}^{j+1}\to \pi(\hat{\Omega}^{j+1})$ is a homotopy equivalence. Let $\omega\in \pi(\hat{\Omega}^{j+1}),$ then there exists $t_{\omega }>0$ such that $(\omega,t_{\omega})\in \hat{\Omega}^{j+1}.$ Since $\hat{\Omega}^{j+1}$ is open, then there exists an open neighboroud $U_{\omega }\times (t_{1},t_{2})$ of $(\omega,t)$ in $\hat{\Omega}^{j+1};$ in particular for every $\eta \in U_{\omega}$ we have $(\eta,t_{\omega})\in \hat{\Omega}^{j+1}$ and $\sigma_{\omega}:\eta \mapsto (\eta, t_{\omega})$ is a section of $\pi$ over $U_{\omega}.$ Collating together the different $\sigma_{\omega}$ for $\omega \in \pi(\hat{\Omega}^{j+1}),$ with the help of a partition of unity, we get a section $\sigma:\pi(\hat{\Omega}^{j+1})\to \hat{\Omega}^{j+1}$ of $\pi.$ Since for every $\omega \in \pi(\hat{\Omega}^{j+1})$ the set $\{t\geq 0\, : \, (\omega,t)\in \hat{\Omega}^{j+1}\}$ is an interval, a straight line homotopy gives the homotopy between $\sigma \circ \pi$ and the identity on $\hat{\Omega}^{j+1}.$ This implies $\pi: \hat{\Omega}^{j+1}\to \pi(\hat{\Omega}^{j+1})$ is a homotopy equivalence.  Using the five lemma and the naturality of the commutative diagrams of the long exact sequences of pairs given by $\pi:(\hat{\Omega}^{j+1},\Omega^{j+1})\to (\pi(\hat{\Omega}^{j+1}),\Omega^{j+1})$  we get ($\pi_{|\Omega^{j+1}}=\textrm{Id}_{|\Omega^{j+1}}$):
$$G_{2}^{i,j}=H^{i}(\hat{\Omega}^{j+1},\Omega^{j+1})\simeq H^{i}(\pi(\hat{\Omega}^{j+1}),\Omega^{j+1}).$$
It remains to prove that for $j>0$
$$\pi(\hat{\Omega}^{j+1})=\Omega_{V}^{j}.$$
First suppose that $(\omega, t)\in \hat{\Omega}^{j+1}.$ Then there exists a subspace $W^{j+1}$ of dimension at least $j+1$ such that $\bar p(\omega,t)|_{W^{j+1}}>0.$ Then $$\omega p|_{W^{j+1}\cap V}=\bar p(\omega,t)|_{W^{j+1}\cap V}>0$$ and by Grassmann formula
$$\dim (W^{j+1}\cap V)=\dim(W^{j+1})+\dim(V)-\dim(W^{j+1}+V)\geq j$$
which implies $\ii_{V}^{+}(\omega)\geq j,$ i.e. $\pi(\omega,t)\in \Omega_{V}^{j}.$ Thus
$$\pi(\hat{\Omega}^{j+1})\subset \Omega_{V}^{j}.$$
Now let $\omega$ be in $\Omega_{V}^{j};$ we prove that there exists $t>0$ such that $\ii^{+}_{h}(\omega,t)\geq j+1.$
Since $\omega \in \Omega_{V}^{j}$ then there exists a subspace $V^{j}\subset V$ of dimension at least $j$ such that
$$\omega p|_{V^{j}}>0.$$ Fix a scalar product on $\R^{n+1}$ and let $e\in \R^{n+1}$ be such that $V^{\perp}=span \{e\};$ consider the space $W=\{\lambda e\}_{\lambda\in \R}+V^{j},$ whose dimension is at least $j+1$ since $e\perp V^{j}\subset V.$ Then the matrix for $\bar{p}_{h}(\omega,t)|_{W}$ with respect to the fixed scalar product has the form:
$$Q_{W}(\omega,t)=\left(\begin{array}{cc}\omega a_{0}+t &{}^{t}\omega a\\
\omega a & \omega Q_{V^{j}}\end{array}\right)$$
where $\omega Q_{V^{j}}$ is the matrix for $\bar{p}(\omega,t)|_{V^{j}}=\omega p|_{V^{j}}.$ Since $\omega p|_{V^{j}}>0$ we have that for $t>0$ big enough $\det(Q_{W}(\omega,t))=t \det(\omega Q_{V^{j}})+\det (\begin{smallmatrix}\omega a_{0}&\omega a\\\omega a&\omega Q_{V^{j}}\end{smallmatrix})$ has the same sign of $\det(\omega Q_{V^{j}})>0.$ For such a $t$ we have
$$\bar{p}_{h}(\omega,t)|_{W}>0$$ and since $\dim(W)\geq j+1$ this implies $(\omega,t)\in \hat{\Omega}^{j+1}$ and $\omega \in \pi(\hat{\Omega}^{j+1}).$ Thus
$$\Omega_{V}^{j}\subset \pi(\hat{\Omega}^{j+1})$$ and this concludes the proof.

\end{proof}

\section{Remarks on higher differentials and examples}

Let $X\subset \P^{n}$ be a compact, locally contractible subset and consider the two inclusions:
$$X\stackrel{j}{\longrightarrow}\P^{n}\quad \textrm{and}\quad \P^{n}\backslash X\stackrel{c}{\longrightarrow}\P^{n}.$$
We recall the existence for every $k\in \Z$ of the following exact sequence, which is a direct consequence of Alexander-Pontryagin Duality:
$$0\rightarrow\ker (c_{*})\rightarrow H_{k}(\P^{n}\backslash X)\stackrel{c_{*}}{\rightarrow}H_{k}(\P^{n})\simeq H^{n-k}(\P^{n})\stackrel{i^{*}}{\rightarrow}H^{n-k}(X)\rightarrow \textrm{coker} (i^{*})\rightarrow 0$$
In particular we have the following equality for the $k$-th $\Z_{2}$-Betti number of $\P^{n}$:
\begin{equation}\label{eqrk}
b_{k}(\P^{n})=\textrm{rk}(c^{*})_{k}+\textrm{rk}(j_{*})_{n-k}
\end{equation}
Consider now $p:\R^{n+1}\to \R^{k+1}\supseteq K$ such that $$\ii^{+}(\bar{p}\eta)=\mu\quad \forall \eta \in \Omega.$$
Then in this case $\Omega^{1}=\cdots=\Omega^{\mu}=\Omega$ and $\Omega^{\mu+1}=\cdots =\Omega^{n+1}=\emptyset.$ For \emph{any} scalar product $g$ on $\R^{n+1}$ we have $D_{\mu}=\Omega^{\mu}=\Omega$ and we denote by $w_{k,\mu}$ the $k$-th Stiefel-Whitney class of the $\R^{\mu}$-bundle $\bar{p}^{*}\mathcal{L}_{j}^{+}\to \Omega$ (the class now defined is independent from $g$ and also are the following results). We define $\gamma_{k,\mu}\in H^{k+1}(C\Omega,\Omega){\simeq} H^{k}(\Omega)$ by
$$\gamma_{k,\mu}\doteq\partial^* w_{k,\mu}$$
(notice that this notation agrees with the one previously used for $\gamma_{1,j}$).\\
Letting $(E_{r},d_{r})$ be the spectral sequence of Theorem A convergent to $H_{n-*}(X),$ where as usual $X=p^{-1}(K)\subseteq \P^{n},$ we have that $(E_{r},d_{r})$ degenerates at $(k+2)$-th step and $E_{2}=\cdots=E_{k+1}$. Moreover $E_{k+1}$ has entries only in the $0$-th and the $(k+1)$-th column: $$E_{k+1}^{a,b}=\left\{ \begin{array}{cc} \Z_{2}&\textrm{if $a=0$ and $\mu\leq b\leq n$ or}\\
& \textrm{$a=k+1$ and $0\leq b<\mu$} \\ 0& \textrm{otherwise}\end{array}\right.$$
Thus the only possible nonzero differential is $d_{k+1},$ for which we prove the following.

\begin{teo}\label{bundle}Suppose $\ii^{+}\equiv \mu.$ Then $E_{2}=\cdots=E_{k+1}$ and the only possible nonzero differential is $d_{k+1}:E_{k+1}^{0,b}\to E_{k+1}^{k+1,b-k}$ for $\mu\leq b\leq n$ and it is given by:
$$d_{k+1}(x)=x\smile \gamma_{k,\mu}$$\end{teo}
\begin{remark} Notice that $\gamma_{k,\mu}$ and $x$ are nothing but numbers modulo $2,$ thus since $E_{k+1}^{0,b}= \Z_{2}=E_{k+1}^{k+1,b-k}$ the element $d_{k+1}(x)$ is nothing but the product $x\gamma_{k,\mu}.$
\end{remark}
\begin{proof}
By Theorem C we have that $d_{k+1}:E_{k+1}^{0,b}\to E_{k+1}^{k+1,b-k}$ is identically zero if and only if $\textrm{rk} (j_{*})_{n-b}=1$ and formula (\ref{eqrk}) implies
$$(d_{k+1})_{0,b}\equiv 0 \quad \textrm{iff}\quad \textrm{rk} (c^{*})_{b}=0$$
where $c^{*}$ is the map induced by $c:\P^{n}\backslash X\hookrightarrow \P^{n}.$
Consider now the following commutative diagram:
$$\begin{tikzpicture}[xscale=3, yscale=2.5]

    \node (A0_0) at (0, 0) {$\P^{n}\backslash X$};
    \node (A1_0) at (1, 0) {$\P^{n}$};
    \node (A0_1) at (0, 1) {$B$};
    \node (A1_1) at (1, 1) {$\Omega \times \P^{n}$};
    \path (A0_0) edge [->] node [auto] {$c$} (A1_0);
    \path (A0_1) edge [->] node [auto,swap ] {${\beta_{r}}_{|B}$} (A0_0);
    \path (A1_1) edge [->] node [auto] {$\beta_{r}$} (A1_0);
     \path (A0_1) edge [->] node [auto] {$\beta_{r}\circ \iota$} (A1_0);
    \path (A0_1) edge [->] node [auto] {$\iota$} (A1_1);

      \end{tikzpicture}
$$
Since ${\beta_{r}}_{|B}$ is a homotopy equivalence, then $$\textrm{rk}(c^{*})_{b}=\textrm{rk}(\iota^{*}\beta_{r}^{*})_{b}.$$
Let $\P^{\mu-1}\hookrightarrow P(\bar{p}^{*}\mathcal{L}_{\mu})\to \Omega$ be the projectivization of the bundle $\R^{\mu}\hookrightarrow \bar{p}^{*}\mathcal{L}_{\mu}\to \Omega.$ It is easily seen that the inclusion $$P(\bar{p}^{*}\mathcal{L}_{\mu})\hookrightarrow B$$ is a homotopy equivalence. From this, letting $l:P(\bar{p}^{*}\mathcal{L}_{\mu})\to \P^{n}$ be the restriction of  $\beta_{r}\circ \iota$ to $P(\bar{p}^{*}\mathcal{L}_{\mu}),$ it follows that:
$$\textrm{rk}(c^{*})_{b}=\textrm{rk}(l^{*})_{b}.$$
Let $y\in H^{1}(\P^{n})$ be the generator; since $l$ is a linear embedding on each fiber, then by Leray-Hirsch, it follows that
$$H^{*}(P(\bar{p}^{*}\mathcal{L}_{\mu}))=H^{*}(\Omega)\otimes \{1,l^{*}y,\ldots,(l^{*}y)^{\mu-1}\}.$$
Thus for $\mu\leq b\leq n$ we have:
\begin{align*}l^{*}y^{b}=&(l^{*}y)^{b}=(l^{*}y)^{\mu}\smile (l^{*}y)^{b-\mu}\\
=&\beta_{l}^{*}w_{k,\mu}\smile(l^{*y})^{\mu-k}\smile (l^{*}y)^{b-\mu}\\
=&\beta_{l}^{*}w_{k,\mu}\smile (l^{*}y)^{b-k}.
\end{align*}
Thus $(d_{k+1})_{0,b}$ is zero if and only if $w_{k,\mu}=0$ and by looking at the definition of $\gamma_{k,\mu}$ we see that
$$d_{k+1}(x)=x\smile \gamma_{k,\mu}.$$
\end{proof}

\begin{example}[The case of one quadric]
This is the most elementary example we can consider, namely the homology of a single quadric in $\P^{n}.$ Let $q\in \Q$ be a quadratic form on $\R^{n+1}$ with signature $(a,b)$ with $a \leq b$ (otherwise we can replace $q$ with $-q$) and $a+b=\textrm{rk}(q)\leq n+1.$ Consider $$X_{a,b}=\{q=0\}\subset \P^{n}.$$ For example, in the case $q$ is nondegenerate (i.e. $a+b=n+1$) then $X_{a,b}$ is smooth and $S^{a-1}\times S^{b-1}$ is a double cover of it.\\
Define the two vectors $h^{-}(X_{a,b}),h^{+}(X_{a,b})\in \N^{n}$ by: $$h^{-}(X_{a,b})=(\underbrace{1,\ldots,1}_{n+1-b},0,\ldots,0),\quad h^{+}(X_{a,b})=(0,\ldots,0,\underbrace{1,\ldots,1}_{a}).$$
Then a straightforward application of Theorem A gives the following identity for the array whose components are the $\Z_{2}$-Betti numbers of $X_{a,b}:$
$$(b_{0}(X_{a,b}),\ldots,b_{n}(X_{a,b}))=h^{-}(X_{a,b})+h^{+}(X_{a,b}).$$
Moreover if we let $j:X_{a,b}\to \P^{n}$ be the inclusion, then Theorem C gives the following:
$$(\textrm{rk}(j_{*})_{0},\ldots,\textrm{rk}(j_{*})_{n})=h^{-}(X_{a,b}).$$

\end{example}

\begin{example}[The case of two quadrics]In the case $p=(q_{1},q_{2})$ and $\ii^{+}$ not constant, then the spectral sequence of Theorem A degenerates at the second step and $E_{2}=E_{\infty}.$ In the case of constant positive index we can use Theorem \ref{bundle} to find $H_{*}(p^{-1}(K))$ (notice that $K\neq \{0\}$ again implies $E_{2}=E_{\infty}.$)

\end{example}
\begin{example}[see \cite{Milnor}]
For $a=1,2,4,8$ consider the isomorphism $\R^{a}\simeq A$ where $A$ denotes respectively $\R, \C,\mathbb{H},\mathbb{O}.$
Consider the quadratic map
$$h_{a}:\R^{a}\oplus \R^{a}\to \R^{a}\oplus \R$$
defined, using the previous identification $\R^{a}\simeq A,$ by
$$(z,w)\mapsto (2z\overline{w},|w|^{2}-|z|^{2}).$$
Then it is not  difficult to prove that $h_{a}$ maps $S^{2a-1}$ into $S^{a}$ by a Hopf fibration. Hence it follows that
$$\emptyset=K_{a}\doteq h_{a}^{-1}(0)\subset \P^{2a-1}.$$
In each case we have $\ii^{+}(\omega h_{a})=a$ for every $\omega \in \Omega=S^{a}.$ Using Theorem \ref{bundle}, since $K_{a}=\emptyset$ then $d_{a+1}$ must be an isomorphism, hence
$$0\neq w_{a,a}=w_{a}(\bar{h}_a^*\mathcal{L}_{a})\in H^{a}(S^{a}).$$
For example in the case $a=2$ we have the standard Hopf fibration ${h_{2}}_{|S^{3}}:S^{3}\to S^{2}$ and the table for $E_{2}=E_{3}$ is:
$$
\begin{array}{|c|c|c|c}
\Z_2&0&0&0\\
\Z_{2}&0&0&0\\
0&0&0&\Z_{2}\\
0&0&0&\Z_2\\
\hline
\end{array}
$$
The bundle $\R^{2}\hookrightarrow \bar{h}_{a}^{*}\mathcal{L}_{2}\to S^{2}$ has total Stiefel-Whitney class $$w(\bar{h}_{a}^{*}\mathcal{L}_{2})=1+w_{2,2},\quad w_{2,2}\neq 0$$
and the differential $d_{3}$ is an isomorphism.\\
Notice that for $a=1,2,4,8$ we have $\ker(\omega h_{a})=0$ for every $\omega \in \Omega.$ It is an interesting fact that the contrary also is true.
\begin{fact}if $p:\R^{m}\to \R^{l}$ is such that $\ker(\omega p)=\{0\}$ for every $\omega\in S^{l}$ and  ${p}_{|S^{m-1}}:S^{m-1}\to S^{l-1}$ then, up to orthonormal change of coordinates $p=h_{a}$ for some $a\in\{1,2,4,8\}.$
\end{fact}
\begin{proof}
First observe that $\ii^{+}\equiv c$ for a  constant $c$ and that $m=2c.$ Then, since $p$ maps the sphere $S^{sc-1}$ to the sphere $S^{l-1},$ we have $$\emptyset=p^{-1}(\{0\})\subset \P^{sc-1}.$$
Thus Theorem \ref{bundle} implies that the differential $d_{l}$ must be an isomorphism and this forces $l=c+1.$ Moreover the condition $\ker(\omega p)=\{0\}$ for every $\omega\in S^{c-1}$ says also $p_{|S^{2c-1}}:S^{2c-1}\to S^{c}$ is a submersion. It is a well-known result (see \cite{Yiu}) that the preimage of a point trough a quadratic map between spheres is a sphere, and thus $p_{|S^{2c-1}}$ is the projection of a sphere-bundle between spheres, hence it must be a Hopf fibration.
\end{proof}

The situation in the case $\{\omega \in S^{l-1}\, : \,\ker(\omega p)\neq 0\}=\emptyset$ with only the assumption $X=\emptyset$ (which is weaker than $p(S^{m-1})\subset S^{l-1}$) is more delicate.
\end{example}

\begin{example}
For $i=1,\ldots,l$ let  $p_{i}:\R^{n_{i}}\to \R^{k+1}$ be a quadratic map and set $N=\sum_{i}n_{i}.$ Define the map
$$\oplus_{i}p_{i}:\R^{N}\to \R^{k+1}$$ by the formula
$$(x_{1},\ldots,x_{l})\mapsto \sum_{i=1}^{l}p_{i}(x_{i})\quad x_{i}\in \R^{n_{i}}.$$
Then for every $\omega \in S^{k}$ we have
$$\ii^{+}(\omega(\oplus_{i}p_{i}))=\sum_{i=1}^{l}\ii^{+}(\omega p_{i}).$$
In particular if each $p_{i}$ has constant positive index function with constant value $\mu_{i}$, then $\oplus_{i}p_{i}$ has also constant positive index function with constant value $\sum_{i}\mu_{i}.$\\
Generalizing the previous example, we consider now for $a=1,2,4,8$ the map $h_{a}:\R^{2a}\to\R^{a+1}$ defined above and we take for $n\in \N$ the map
$$n\cdot h_{a}\doteq \oplus_{i=1}^{n}h_{a}:\R^{2an}\to \R^{a+1}.$$
In coordinate the map $n\cdot h_{a}$ is written by:
$$(w,z)\mapsto(2\langle z,w\rangle,\|w\|^{2}-\|z\|^{2}),\quad w,z\in A^{n}.$$
Then for this map we have
$$\ii^{+}\equiv na,\quad\textrm{and}\quad (n\cdot \bar{h}_{a})^{*}\mathcal{L}_{na}=n(\bar{h}_a^*\L_{a})=\underbrace{\bar{h}_a^*\L_{a}\oplus \cdots\oplus \bar{h}_a^*\L_{a}}_{n}$$
The solution of $\{n\cdot h_{a}=0\}$ on the sphere $S^{2a-1}$ is diffeomorphic to the Stiefel manifold of 2-frames in $A^{n},$ and it is a double cover of
$$n\cdot K_{a}\doteq\{n\cdot h_{a}=0\}\subset \P^{2na-1}.$$
We can proceed now to the calculation of the $\Z_{2}$-cohomology of $n\cdot K_{a},$ using Theorem \ref{bundle}: we only need to compute $d_{a+1},$ i.e. $w_{a}(n\bar{h}_a^*\L_{a}).$
Since $w_{a}(\bar{h}_a^*\L_{a})=w_{a,a}\neq 0,$ and $w_{k}(\bar{h}_a^*\L_{a})=0$ for $k\neq 0,k\neq a,$ then we have
$$w_{a}(n\bar{h}_a^*\L_{a})=n \,\textrm{mod}\,2\in \Z_{2}=H^{a}(S^{a}).$$

\end{example}

There are some cases in which the problem of describing the index function can be reduced to a simpler problem; this is the case of a quadratic map defined by a bilinear one. We start noticing the following.
\begin{fact}\label{matrix}
Let $L$ be a $n\times n$ real matrix and $Q_{L}$ be the symmetric $2n\times 2n$ matrix defined by:

$$
Q_{L}=\left(\begin{array}{cc}
0 & L\\
^{t}L &0\\
\end{array} \right)
$$
Then, setting $q_{L}$ for the quadratic form defined by $x\mapsto \langle x,Q_{L}x\rangle$ we have:
$$\ii^{+}(q_{L})=\textrm{rk}(L).$$
\begin{proof}
Let $x=(z,w)\in \R^{2n}\simeq \R^{n}\oplus \R^{n};$ then $Q_{L}\left(\begin{smallmatrix}z\\w\end{smallmatrix}\right)=\left(\begin{smallmatrix}Lw\\^tLz\end{smallmatrix}\right ).$ Hence
$\ker Q_{L}=\ker {}^{t}L\oplus\ker L$ and $$\dim (\ker Q_{L})=2\dim(\ker L).$$ Consider now the characteristic polynomial $f$ of $Q_{L}$: $$f(t)=\det (Q_{L}-t I)=\det(t^{2}I -{}^{t}LL)=(-1)^{n}\det ({}^{t}LL-t^{2}I)=(-1)^{n}g(t^{2})$$ where $g$ is the characteristic polynomial of ${}^{t}LL.$ Let now $\lambda\in \R$ be such that $g(\lambda)=0;$ since ${}^{t}LL\geq 0,$ then $\lambda \geq 0$ and $f(\pm \sqrt{\lambda})=0.$ Since $Q_{L}$ is diagonalizable, then for each one of its eigenvalues algebraic and geometric multiplicity coincide, hence
$$\ii^{+}(q_{L})=\ii^{-}(q_{L})=\frac{1}{2}\textrm{rk}(Q_{L}).$$
It follows that
$$\ii^{+}(q_{L})=\frac{1}{2}(2n -\dim( \ker Q_{L}))=\textrm{rk}(L).$$
\end{proof}
\end{fact}
In particular if $b:\R^{n}\times \R^{n}\to \R^{k+1}$ is a bilinear antisymmetric map whose components are defined by $$(x,y)\mapsto \langle (x,y),\left(\begin{smallmatrix}0&B_{i}\\{}^{t}B_{i}&0\end{smallmatrix}\right)(x,y)\rangle$$ for certain real squared matrices $B_{i},\,i=1,\ldots,k+1,$ then we can consider the quadratic map $$p_{b}:\R^{2n}\to \R^{k+1}$$ defined by $(x,y)\mapsto b(x,y).$
In this case we define for $\omega\in S^{k}$ the matrix $\omega B$ by
$$\omega B= \omega_{1}B_{1}+\cdots+\omega_{k+1}B_{k+1}.$$
By the previous fact we have
$$\ii^{+}(\omega p_{b})=\textrm{rk}(\omega B).$$

\begin{example}
Let $\R^{8}$ be identified with the space of pairs of $2\times 2$ real matrices. We apply the previous consideration to describe the topology of
$$\Gamma=\{(X,Y)\in \R^{8}\, : \, [X,Y]=0\}.$$
Since the equation for $\Gamma$ are homogeneous, it is a cone, and we can study the homology of its projectivization $$\P(\Gamma)\subset \P^{7}.$$
If we define $V=\{(X,Y)\in \R^{8}\, : \, \textrm{tr}(X)=\textrm{tr}(Y)=0\}$ and $\Gamma_{V}=\Gamma \cap V,$ then it is readily seen that
$$\Gamma=\Gamma_{V}\oplus \R^{2}.$$
We proceed first to the computation of $H_{*}(\P(\Gamma_{V}))$ using the above theorems.\\
In coordinates $(X,Y)=(\left(\begin{smallmatrix}x&y\\z&-x\end{smallmatrix}\right),\left(\begin{smallmatrix}w&t\\s&-w\end{smallmatrix}\right))$ we have
$$\{[X,Y]=0\}\cap V=\{tz-ys=xt-yw=sx-wz=0\}.$$
Consider the following matrices
$$B_{1}=\left(\begin{array}{ccc}0&0&0\\0&-1&0\\0&0&1\end{array}\right),\quad B_{2}=\left(\begin{array}{ccc}0&0&1\\-1&0&0\\0&0&0\end{array}\right),\quad B_{3}=\left(\begin{array}{ccc}0&1&0\\0&0&0\\-1&0&0\end{array}\right)$$
and the bilinear map $b:\R^{3}\times \R^{3}\to \R^{3}$ whose components are $(x,y)\mapsto \langle x, B_{i}y\rangle.$
Then  $p_{b}:V\to \R^{3}$ equals the quadratic map defined by $(X,Y)\mapsto [X,Y]$ (we are using the above notations for the quadratic map $p_{b}$ defined by a bilinear map $b$). It follows that $$\Gamma_{V}=V\cap \Gamma=\{p_{b}=0\}.$$ Using $\omega B$ for the matrix $\omega_{1}B_{1}+\omega_{2}B_{2}+\omega_{3}B_{3},$ then by the previous fact we have
$$ \ii^{+}(\omega p_{b})=\textrm{rk}(\omega B)\,\,\, \forall \omega \in S^{2}.$$
Let $\omega Q_{b}$ the symmetric matrix associated to $\omega p_{b}$ by the rule $(\omega p_{b})(x)=\langle x,\omega Q_{b}x\rangle.$ Then $$\omega Q_{b}=\left (\begin{array}{cc}0&\omega B\\{}^{t}\omega B&0\end{array}\right)$$
The matrix $\omega B,$ for $\omega=(\omega_{1},\omega_{2},\omega_{3})\in S^{2}$ has the following form:
$$\left(\begin{array}{ccc}0&\omega_{3}&\omega_{2}\\-\omega_{2}&-\omega_{1}&0\\-\omega_{3}&0&\omega_{1}\end{array}\right)$$
and we immediatly see that $\textrm{rk}(\omega B)=2$ for $\omega \neq 0;$ this gives $$\ii^{+}(\omega p_{b})=2\quad \forall \omega \in S^{2}.$$
Since $\ii^{+}\equiv 2,$ we can apply Theorem \ref{bundle}; letting $(E_{r},d_{r})$ be the spectral sequence of Theorem A converging to $H_{n-*}(\P(\Gamma_{V})),$ we have the following picture for $E_{2}=E_{3}:$
$$
\begin{array}{|c|c|c|c}
\Z_2&0&0&0\\
\Z_2&0&0&0\\
\Z_2&0&0&0\\
\Z_{2}&0&0&0\\
0&0&0&\Z_{2}\\
0&0&0&\Z_2\\
\hline
\end{array}
$$
Consider the section $\sigma:S^{2}\to S^{2}\times \R^{6}$ defined for $\omega=(\omega_{1},\omega_{2},\omega_{3})\in S^{2}$ by:
$$\sigma(\omega)=(\omega_{2},0,\omega_{1},-\omega_{1}\omega_{3},\omega_{2}\omega_{3},\omega_{1}^{2}+\omega_{2}^{2}).$$
Since for every $\omega\in S^{2}$
$$(\omega Q_{b})\sigma(\omega)=\sigma(\omega)$$
then it follows that $\sigma$ is  a section of the bundle $\bar{p}_{b}^{*}\L_{2}.$ The index sum of the zeroes of $\sigma$ (which occur only at $(0,0,1),(0,0,-1)\in S^{2}$) is even, thus the euler class $e$ of $\bar{p}_{b}^{*}\L_{2}$ is even. This implies
$$w_{2}(\bar{p}_{b}^{*}\L_{2})=e \,\textrm{mod}\,2=0.$$
Thus by Theorem \ref{bundle} we have $d_{3}\equiv 0$ and $E_{2}=E_{3}=E_{\infty}.$ It follows that the only nonzero homology groups of $\P(\Gamma_{V})$ are:
$$H_{0}(\P(\Gamma_{V}))=H_{3}(\P(\Gamma_{V}))=\Z_{2}\quad \textrm{and}\quad H_{1}(\P(\Gamma_{V}))=H_{2}(\P(\Gamma_{V}))=(\Z_{2})^{2}.$$
Actually since the equations for $\Gamma_{V}$ are given by the vanishing of the minors of the matrix $\left(\begin{smallmatrix}x&z&y\\w&s&t\end{smallmatrix}\right),$ then $\Gamma_{V}$ is the Segre variety $\Sigma_{2,1}\simeq \P^{1}\times \P^{2}.$\\
Notice also that in the case $\ii^{+}\equiv \mu$ if we take $l_{v}^{+}=\{t^{2}v\}_{t\in\R}$, then we can easily calculate the homology of $X_{l_{v}^{+}}=\{x\in \P^{n}\, : \,p(x)\in l_{v}^{+}\})$ (the preimage of a half line): using Theorem 2 we immediatly see that $E_{2}=E_{\infty}$ which implies $H_{*}(X_{l_{v}^{+}})\simeq H_{*}(\P^{n-\mu}).$
\end{example}
\begin{example}
Consider the map $p:\R^{4}\to \R^{3}$ given by
$$(x_{0},x_{1},x_{2},x_{3})\mapsto(x_{0}x_{2}-x_{1}^{2},x_{0}x_{3}-x_{1}x_{2},x_{1}x_{3}-x_{2}^{2}).$$
Then $C=\{p=0\}\subset \P^{3}$ is the rational normal curve, the so called twisted cubic. In this case $\Omega=S^{2}$ and the set $\{\omega \in \Omega\, : \, \ker (\omega p)\neq 0\}$ consists of two disjoint ovals in $S^{2},$ bounding two disks $B_{1}, B_{2}.$ Then $S^{2}$ is the disjoint union of the sets $\textrm{Int} (B_{1}), \partial B_{1},R,\partial B_{2},\textrm{Int}(B_{2}),$ on which the function $\ii^{+}$ is constant with value respectively $2,1,2,2,2.$ Then
$$\Omega^{1}=S^{2},\quad \Omega^{2}=S^{2}\backslash \partial B_{1}, \quad \Omega^{3}=\emptyset$$
and the second term of the spectral sequence $(E_{r},d_{r})$ of Theorem A converging to $H_{3-*}(C)$ is the following:
$$
\begin{array}{|c|c|c|c}

\Z_2&0&0&0\\
\Z_{2}&0&0&0\\
0&\Z_{2}&0&0\\
0&0&0&\Z_2\\
\hline
\end{array}
$$

The differential $d_{2}:E_{2}^{1,1}\to E_{2}^{3,0}$ is an isomorphism; hence $E_{3}=E_{\infty}$ has the following picture:

$$
\begin{array}{|c|c|c|c}

\Z_2&0&0&0\\
\Z_{2}&0&0&0\\
0&0&0&0\\
0&0&0&0\\
\hline
\end{array}
$$

From the previous, using Theorem C, we see that $j_{*}:H_{1}(C)\to H_{1}(\P^{3})$ is an isomorphism (we can check this fact also by noticing that, since $C$ is a curve of degree $3,$ then the intersection number of $C$ with a generic hyperplane $H\subset \P^{3}$ is odd).
\end{example}

\end{document}